\documentclass[a4paper,11pt]{article}
\usepackage{a4wide}
\usepackage{multicol, float}
\usepackage{lineno,hyperref}
\usepackage{verbatim, euscript}
\usepackage{bbm,amssymb}
\usepackage{amsfonts}
\usepackage{amsmath, mathtools}
\usepackage{pictexwd}
\usepackage{amsthm}
\usepackage{graphicx}
\usepackage{color}
\usepackage{mathrsfs}
\usepackage{bookmark}

\usepackage{mathtools}
\mathtoolsset{showonlyrefs}

\DeclarePairedDelimiter\floor{\lfloor}{\rfloor}
\DeclarePairedDelimiter\norm{\lVert}{\rVert}
\DeclarePairedDelimiter\abs{|}{|}

\DeclareMathOperator{\Card}{Card}
\DeclareMathOperator{\Leb}{Leb}

\DeclareMathOperator{\supp}{supp}

\DeclareMathOperator{\CPP}{CPP}

\theoremstyle{plain}
\newtheorem{thm}{Theorem}
\newtheorem{lem}{Lemma}[section]
\newtheorem{cor}[lem]{Corollary}
\newtheorem{prop}[lem]{Proposition}

\theoremstyle{definition}
\newtheorem{rem}[lem]{Remark}
\newtheorem{defi}[lem]{Definition}

\newcommand{\eps}{\varepsilon}
\newcommand{\E}{\mathbb{E}}
\newcommand{\PP}{\mathbb{P}}
\newcommand{\1}{\mathbbm{1}} 
\newcommand{\N}{\mathbb{N}}

\newcommand{\R}{\mathbb{R}}
\newcommand{\Q}{\mathbb{Q}}

\providecommand{\U}{}
\renewcommand{\U}{\mathcal{U}}
\newcommand{\EE}[1]{\mathbb{E} \left[ #1 \right]}

\newcommand{\pp}[1]{\mathbb{P} \left( #1 \right)}

\newcommand{\QQ}{\mathbb{Q}}
\renewcommand{\epsilon}{\eps}
\renewcommand{\emptyset}{\varnothing} % Better looking empty sets

\title{The genealogy of nearly critical branching processes\\ in varying environment}

\usepackage{authblk}

\newlength{\affilskip}
\author[1]{Florin Boenkost}
\author[2]{Félix Foutel-Rodier}
\author[1]{Emmanuel Schertzer}

\affil[1]{
    Faculty of Mathematics, University of Vienna,\authorcr Oskar-Morgenstern-Platz 1, 1090 Wien, Austria
    \vspace{\affilskip}
}
\affil[2]{
    Université Paris Cité, CNRS, MAP5, F-75006 Paris, France
    \vspace{\affilskip}
}

\begin{document}
	
\maketitle
\begin{abstract}
	Building on the spinal decomposition technique in
	\cite{FoutelSchertzer22} we prove a Yaglom limit law for the rescaled
	size of a nearly critical branching process in varying environment
	conditional on survival. In addition, our spinal approach allows us to
	prove convergence of the genealogical structure of the population at a
        fixed time horizon -- when the sequence of trees are envisioned as
	a sequence of metric spaces --  in the Gromov--Hausdorff--Prohorov (GHP)
	topology. We characterize the limiting metric space as a time-changed
	version of the Brownian coalescent point process \cite{Popovic2004}.
	
	Beyond our specific model, we derive several general results
        allowing one to go from spinal decompositions to
        convergence of random trees in the GHP topology. As a direct
        application, we show how this type of convergence naturally
        condenses the limit of several interesting genealogical
        quantities: the population size, the time to the most-recent
        common ancestor, the reduced tree, and the tree generated by $k$
        uniformly sampled individuals.  As in a recent article by the
        authors \cite{FoutelSchertzer22}, we hope that our specific
        example illustrates a general methodology that could be applied
        to more complex branching processes.
\end{abstract}

\section{Introduction}
Describing the genealogies of different population models as the size of
the population gets large is a prominent task in the theory of
mathematical population genetics as well as in the study of branching
processes, and goes back at least to the two seminal papers of Kingman
\cite{Kingman1982a, Kingman1982}. As of today, there is a vast
literature on the genealogies of various types of population models, see
for example \cite{Blancas2022,Grosjean2018,Kersting2022,Moehle2001,Popovic2004,Schweinsberg2003}.
Here, among other topics, we focus on the scaling limit of the
genealogy of nearly critical branching processes in varying environment
conditional on survival, see Section~\ref{Sec:Model} for a
precise definition.

The scope of this paper is two-fold. On the one hand, we prove a
\emph{Yaglom-type} result \cite{Yaglom1947} for the size of a class
of nearly critical branching processes in varying environment conditioned
on survival for a long time, as well as a result on the scaling limit of
their genealogy. For this purpose we prove a \emph{Kolmogorov-type
asymptotics} \cite{Kolmogorov1938} for the probability that the
process survives for a long time, which might be interesting in its own
right. On the other hand, our objective is to use this specific model to
illustrate the techniques developed in \cite{FoutelSchertzer22,
Greven2009, Gloede2013, gromov2007metric}. Our approach relies on new
general results on random metric spaces that could be applied widely to
branching processes with critical behavior.
The \emph{spinal decomposition} method in
\cite{FoutelSchertzer22} enables one to study genealogies by
showing convergence of random metric measure spaces in the Gromov-weak
topology using a method of moments. To prove convergence of the $k$-th
moment of a genealogy, a \emph{many-to-few} formula allows us to reduce this
problem to calculating the distribution of a tree with $k$ leaves, the
so-called $k$-spine tree. The distribution of this $k$-spine tree is
obtained through a random change of measure relying on choosing a
suitable ansatz for the genealogy of the process. In order to make the
$k$-spine method work one would require the offspring distribution to
have moments of all order, which is far from the optimal condition.
However, using a truncation argument, we are able to work under a second
moment condition. Finally, we show how to reinforce the Gromov-weak
convergence of the genealogy to a Gromov--Hausdorff--Prohorov
convergence. As a direct consequence of our main
result, we obtain the convergence of many interesting genealogical quantities.

In summary, a part of the current work is dedicated to proving general
results on random metric spaces. It is only at the end of the
article that this general formalism will be used to treat our
specific problem as a (semi-)direct application. Along the way, we will
obtain new criteria to obtain a Yaglom law for branching processes in
varying environment (see Remark~\ref{rem:comparisonAssumptions} below).
We hope that this application will help convince the reader that the
abstract formalism introduced in \cite{FoutelSchertzer22, Greven2009,
Gloede2013, gromov2007metric} (among others) and further developed
here is particularly well suited when treating general branching
processes near criticality.

\subsection{Relevant literature}

The genealogical structure of branching processes is a topic that
has drawn considerable attention, particularly in the critical case.
We give a non-exhaustive overview of some results related to
our work. A first class of results focuses on the so-called
\emph{reduced process} \cite{Fleischmann1977, Yakymiv1981,
OConnell1995}, that counts the number of individuals having
descendants at a fixed generation in the future. The limit
of the reduced process is often expressed as a (time-inhomogeneous) pure
birth process, see \cite{Kersting2022} for results in this direction for
branching processes in varying environment. Second, in the same spirit as
the theory of exchangeable coalescents, there has been an interest in
describing the genealogy obtained by sampling a fixed number $k$ of
individuals uniformly at a given time \cite{Johnston2019, Harris2020,
Lambert2018}. We also want to point at \cite{Lambert2010} for an
explicit construction of the genealogy of splitting trees as a discrete
coalescent point process, which plays a central role in
\cite{Lambert2018}.

In this work, we show how the convergence of the reduced process,
of the genealogy of a uniform sample, and of the population size can be
deduced by viewing the population as a random ultrametric measure space
converging in the Gromov--Hausdorff--Prohorov topology. Encoding trees as
random metric spaces is a long-standing idea that has proved fruitful
both for constant-size population models \cite{Evans2000, Greven2009} and
for branching processes \cite{Aldous1993, Depperschmidt2019}. Let us
point out the following difference between our approach and more common
convergence results in the Gromov--Hausdorff--Prohorov topology such as
\cite{Aldous1993, LeGall2002, Miermont2008}. Here, we consider the
genealogy of the population at a fixed time horizon, whereas many works
\cite{Aldous1993, LeGall2002, Miermont2008} are interested in the limit
of the whole genealogical tree. Although in the limit one can construct
the genealogy at a fixed time (the Brownian coalescent point process)
from that of the whole tree (the Brownian continuum random tree),
see \cite{Popovic2004}, it is not straightforward to go from the
convergence of the latter object to our type of convergence
result. Also, from a technical point of view, the convergence of
the whole tree is typically obtained by studying a height
function, whereas our approach relies on a $k$-spine method.

Spinal decompositions techniques are widely used in the
theory of branching processes. In particular, the case of a single spine has
been proven to be an important tool, see \cite{Lyons95, shi_2015,
Geiger1999}. The classical Yaglom limit and the Kolmogorov estimate for
critical Galton--Watson processes are obtained in \cite{Lyons95} using this
technique. Here, we adapt the one-spine decomposition from \cite{Lyons95} to
prove the Kolmogorov asymptotics for branching processes in varying
environment, see Theorem~\ref{thm:kolmogorov}. These techniques have been
further developed -- for example in \cite{Harris2017, FoutelSchertzer22}
-- to allow for a $k$-spine construction, which in turn allows us to
compute higher moments of the tree by a random change of measure. Such a
$k$-spine construction has also been used to extract information about
the genealogy of several branching processes \cite{Harris2020,
Johnston2019, FoutelSchertzer22}, and the current work follows a similar
path. We also want to point at \cite{Gonzalez2022, Harris2021} for a
computation of higher order moments for general spatial branching
processes, and the derivation of a Yaglom law in this context.

For an introduction to branching processes in varying environment we
refer to the monograph of Kersting and Vatutin \cite{Kersting2017Book}.
The work of Kersting \cite{Kersting2020} gives a complete
classification of critical branching processes in varying environment
and proves a Yaglom limit result in this case. Later, Cardona-Tob\'{o}n
and Palau \cite{CardonaTobon2021} managed to prove the Yaglom-type result
for critical branching processes in varying environment by applying a
two-spine decomposition and therefore providing a probabilistic proof of
the result obtained by analytical methods in \cite{Kersting2020}. They
make use of the remarkable fact that the distribution of the time to the
most-recent common ancestor of two individuals can be found explicitly.
Recently, Kersting \cite{Kersting2022} managed to obtain the asymptotic
distribution of the time to the most-recent common ancestor of all
individuals living at some large generation $n$. Our work is a
continuation and extension of these results, but does not build strongly
on them, since we are dealing with nearly critical processes instead.

While completing this article we became aware of the works of
Harris, Palau and Pardo \cite{Harris2022} and of Conchon-{-}Kerjan, Kious
and Mailler \cite{Conchon2022}. In the former work, building on
techniques of theirs developed in \cite{Harris2017} and
\cite{CardonaTobon2021}, the authors give a forward in time construction
of the $k$-spine for a critical branching process in varying environment.
Subsequently, via a change of measure relating the original branching
process in varying environment and the $k$-spine tree, they obtain the
asymptotic distribution of the splitting times of the tree spanned by $k$
individuals chosen uniformly at large times, which bares similarities
to our Corollary~\ref{cor:three_consequences}~(iii). Their results and the
spinal decomposition are related to our work and are in the same spirit.
However, due to the different points of view and approaches in their work
and ours, we believe that these two articles are complementary to each
other. We refer to Remark~\ref{rem:comparisonAssumptions} and
Remark~\ref{rem:rho} for a more detailed comparison. In the latter work
\cite{Conchon2022}, the authors derive the scaling limit of the whole
tree structure of a branching process in i.i.d.\ random environment for
the Gromov--Hausdorff--Prohorov topology. Their main result is consistent
with ours in the sense that, under their hypothesis, the limit that we
find (the Brownian coalescent point process) corresponds to the reduced
tree at a given time of their limit (the Brownian continuum random tree).
Nevertheless, as mentioned above, our type of result is not directly
implied by the convergence of the whole tree and the techniques are
different. Also, our hypothesis can be seen as more general, in the sense
that we do not require the environment to have properties of i.i.d.\
sequences, and consider a sequence of nearly critical environments rather
than a single strictly critical environment. In particular, stationary
environments (and thus i.i.d.\ environments) can only lead to the
Brownian coalescent point process in the limit, whereas we recover any
time-change of a Brownian coalescent point process, including non-binary
trees whenever the limiting variance process has jumps. Finally, let us
point to \cite{Greven2009Genealogy} which considers a binary branching
process in a random environment, modelling the effect of a catalyst that
modifies the branching rate of the population. They provide the scaling
limit of the genealogy of this branching process, recovering in the limit
a time-changed version of a coalescent point process similar to the one
arising here. Our work goes way beyond the case of binary branching
processes covered in \cite{Greven2009Genealogy}, and in a sense extends
this result to a much broader universality class.

\subsection{Outline}
The outline of the manuscript is as follows. First, in Section~\ref{Sec:Model}
we precisely define the notion of a nearly critical branching process
in varying environment and provide some examples. We also state our two main
results, the Kolmogorov asymptotics (Theorem~\ref{thm:kolmogorov}) and the
convergence of the genealogy (Theorem~\ref{thm:yaglom}), and
in Corollary~\ref{cor:three_consequences} we state some straightforward consequences of
our Gromov--Hausdorff--Prohorov convergence.
In Section~\ref{sec:gromovTopologies} we introduce the topologies on
metric measure spaces needed in this work, and prove some continuity
results for ultrametric spaces. In Section~\ref{sec:CPP}, we introduce the limiting object for the genealogy, the \emph{environmental coalescent
point process} and derive some of its properties. 
The final sections are devoted to the proofs of our results. In
Section~\ref{sec:spinalDecompositions} we recall the spinal decomposition
tools for branching processes that we need. The proof of the Kolmogorov
asymptotics (Theorem~\ref{thm:kolmogorov}) is carried out in
Section~\ref{sec:kolmogorov_proof}, and that of the convergence of the
genealogy (Theorem~\ref{thm:yaglom}) in Section~\ref{Sec:Proofs}. 
Lastly, some results of technical nature are given in Appendix~\ref{sec: Appendix}.

\section{Model} \label{Sec:Model}
Branching processes in varying environment (BPVE for short) are a natural
generalization of classical Galton--Watson processes in which 
the offspring distribution can change between generations. 
Throughout this work we will try to follow the notation of
\cite{Kersting2017Book}.

We begin by defining formally the notion of a branching process in
varying environment. Denote by $\mathcal{P}(\N_0)$ the space of
probability measures on $\N_0=\{0,1,2,\dots\}$ and let $v=(f_1,f_2,\dots)$ be a sequence
of probability measures on $\N_0$. 
\begin{defi}
        We call a process $(Z_n, n \geq 0)$ a BPVE with environment
        $v=(f_1,f_2,\dots)$, with $v \in \mathcal{P}(\N_0)^{\N}$, if $(Z_n,
        n \geq 0)$ has the representation
	\begin{align}
		Z_n = \sum_{i=1}^{Z_{n-1}} \xi_{i,n}, \quad Z_0=1,
	\end{align}
        where for each $n \geq 1$ the random variables $\{\xi_{i,n}, i
        \geq 1 \}$ are independent and distributed as $f_n$.
\end{defi}
Note that we assume that the BPVE starts from a single
individual at generation $0$. Subsequently, it will be
useful to identify a measure $f\in \mathcal{P}(\N_0)$ with its
generating function
\begin{align}
	f(s)= \sum_{k=0}^{\infty} s^k f[k], \qquad s\in [0,1],
\end{align}
where $f[k]$, $k \geq 0$, denotes the weight of the distribution $f$ at $k$.
Note that we  will use the same symbol $f$ for the distribution as well
as for the corresponding generating function $f(s)$. Recall that the
first and second factorial moments of a random variable with distribution
$f$ can be expressed as
\begin{align}
	f'(1)= \sum_{k=1}^{\infty} k f[k], \qquad f''(1)= \sum_{k=2}^{\infty} k (k-1) f[k]. 
\end{align}
In the following, we consider a sequence $Z^{(N)}=\big( Z_n^{(N)}, n
\geq 0\big)$ of branching processes in varying environment indexed by
$N$. That is, we consider a sequence of environments
$v^{(N)}=(f_1^{(N)},f_2^{(N)},\dots)$ and study the asymptotic
probability that this process survives $N$ generations as
well as the asymptotic genealogy of the population at large times.

Due to the fact that we are working with a sequence of branching processes, all
the quantities introduced so far as well as the quantities which will be
introduced in the subsequent sections depend on $N$. Additionally, all
expectations and probabilities will depend on $N$. For the convenience of
the reader and to avoid overloading the notation, we will drop
the superscript $N$ most of the times. The reader should keep the
implicit dependency on $N$ in mind. Furthermore, whenever a sum
is ranging from $1$ to $\floor{Nt}$, we will abuse the notation and simply
write $Nt$ as the upper limit. More generally, whenever no confusion is
possible, we simply write $Nt$ instead of $\floor{Nt}$, as in $f_{Nt}$ or
$\xi_{Nt}$.

Finally, it is usual to encode the genealogy of a branching process as a
random subset of the set of finite words
\[
\U = \{\emptyset\} \cup \bigcup_{n \ge 1} \N^{n}.
\]
A word $u = (u_1, \dots, u_n)$ is interpreted as an individual in the
population living at generation $\abs{u} \coloneqq n$. For two
individuals $u, v \in \U$, we also use the notation $uv$ for their
concatenation, and $u \wedge v$ for their most-recent common ancestor.

We can construct a random tree out of an independent collection
$(K_u, u \in \U)$ of random variables with $K_u \sim f_{\abs{u}+1}$ by
setting 
\begin{equation}
    T = \{ u = (u_1, \dots, u_n) : u_i \le K_{(u_1, \dots,u_{i-1})},\: i \le n \}.
\end{equation}
This construction is carried out more carefully for BPVEs in
\cite[Section~1.4]{Kersting2017Book}. The variable $Z_n$
is recovered as the size of the $n$-th generation of $T$, 
\[
    Z_n = \Card T_n,\quad
    T_n \coloneqq \{ u \in T : \abs{u} = n \},\quad
    n \in \N.
\]

\subsection{Assumptions}
In order to prove the main theorems, Theorem~\ref{thm:kolmogorov} and
Theorem~\ref{thm:yaglom}, we will work under the following assumptions on
the sequence of environments $v=(f_1,f_2,\dots)$, which we call a nearly
critical varying environment. Recall that we do not denote the dependency
of $v$ on $N$ explicitly. Let us introduce the following quantities,
\begin{equation}
	\mu_k= \prod_{i=1}^{k} f_i'(1), \quad \mu_0=1,
\end{equation}
which is the expectation of the BPVE at time $k$, namely $\EE{Z_k}=
\mu_k$. We call a sequence of branching processes in varying environment \emph{nearly critical} if  
\begin{equation}\label{Cond_on_the_environment}
	( \mu_{tN}, t \ge 0 ) \to (e^{X_t}, t \ge 0), \quad \text{ as } N \to \infty, 
\end{equation}
in the Skorohod topology for some càdlàg process $(X_t, t \ge 0)$. 
(Having in mind the important case where the environment $v$ is a
realization of an i.i.d.\ sequence, we use the terminology
process for $(X_t, t \ge 0)$ although it is a deterministic
càdlàg function). The following assumptions ensure that the variance
between the generations does not fluctuate too strongly.
Precisely, for some càdlàg non-decreasing function $\sigma^2:
\R_+ \to \R_+$ and some sequence $\kappa_N \to \infty$, we assume
that for all $t \in \R_+$ such that $\sigma^2$ is continuous at
$t$, 
\begin{equation}\label{assumption_variance}
	\frac{1}{\kappa_N} \sum_{k=1}^{Nt} f_k''(1) \to \sigma^2(t), 
	\quad \text{ as } 
	N \to \infty. 
\end{equation}
We also assume that $\sigma^2$ and $(X_t, t \ge 0)$ do not jump at the
same time. Lastly, we impose the following Lindeberg-type condition 	
\begin{equation} \label{eq:uniform_integrability}
	\forall \epsilon > 0,\quad 
	\frac{1}{\kappa_N} \sum_{k=1}^{Nt} \E[\xi_k^2 \1_{\xi_k \ge \epsilon \sqrt{\kappa_N}}] 
	\to 0, \quad \text{as $N \to \infty$},
\end{equation}
for a generic copy $\xi_k$ distributed as $f_k$. It will turn out that
assumption~\eqref{eq:uniform_integrability} is only required to derive the
asymptotics of the survival probability (the Kolmogorov
asymptotics, see Theorem~\ref{thm:kolmogorov}), and that all
other results only rely on the following weaker assumption
\begin{equation} 
	\tag{\ref*{eq:uniform_integrability}'} \label{eq:lindeberg}
	\forall \epsilon > 0,\quad 
	\frac{1}{\kappa_N} \sum_{k=1}^{Nt} \E[\xi_k^2 \1_{\xi_k \ge \epsilon \kappa_N}] 
	\to 0, \quad \text{as $N \to \infty$},
\end{equation}
which prevents any single individual from having a large number of
offspring of order $\kappa_N$. We believe that given
\eqref{Cond_on_the_environment} and \eqref{assumption_variance} this
assumption, already used in \cite{Borovkov2002}, is nearly optimal. For
instance, in the case of Galton--Watson processes \eqref{eq:lindeberg} is
known to be optimal for convergence to a Feller diffusion when the
initial population size is of order $\kappa_N$ \cite{Grimvall1974}.

\begin{rem} \label{rem:comparisonAssumptions}
	Our definition of near criticality is directly inspired by that for
	Galton--Watson processes, where one generally imposes that the mean and the
	variance of the process converge. An alternative notion of
	criticality was proposed for BPVEs in \cite{Kersting2020} and was further used in \cite{Kersting2022, CardonaTobon2021, Harris2022}. 
        It relies on the behavior of the following two series
	\[
	\mu_{tN} \sum_{k=1}^{tN} \frac{1}{f_k'(1)^2}\frac{f_k''(1)}{\mu_{k-1}} 
	\sim \kappa_N \cdot \int_{[0,t]} e^{X_t - X_s} \sigma^2(ds)
	\to \infty, 
	\quad\text{as $N \to \infty$},
	\]
	and
	\[
	\sum_{k=1}^{tN} \frac{1}{f_k'(1)^2} \frac{f_k''(1)}{\mu_{k-1}}
	\sim \kappa_N \cdot \int_{[0,t]} e^{-X_s} \sigma^2(ds)
	\to \infty,
	\quad\text{as $N \to \infty$},
	\]
	where the asymptotics under assumptions \eqref{Cond_on_the_environment} 
	and \eqref{assumption_variance} are obtained in Lemma~\ref{conv_variance}. 
	Since both series diverge, our BPVE would be classified as critical
	in \cite{Kersting2020}. Assumptions \eqref{Cond_on_the_environment}
	and \eqref{assumption_variance} are not necessary for these series to
	diverge and can therefore be seen as restrictive in that sense.
	
	It should be noted, however, that these stronger assumptions allow us
	to work with sequences of environments and lead to a notion of near
        criticality, whereas the aforementioned works consider a
        fixed environment. Moreover, we only require the mild moment
        condition \eqref{eq:uniform_integrability} on the offspring size,
        which we believe to be nearly optimal, whereas
        \cite{CardonaTobon2021, Harris2022} require a finite third
        moment, and \cite{Kersting2020, Kersting2022} a stronger second
        moment assumption. In particular, if
        \eqref{Cond_on_the_environment} and \eqref{assumption_variance}
        hold, the latter condition (assumption (B) in
        \cite{Kersting2020}) implies that for any sequence $K_N \to
        \infty$ 
	\begin{equation} \label{eq:kerstingConsequence}
		\lim_{N \to \infty} \frac{1}{\kappa_N} \sum_{i=1}^{tN} 
                \E[\xi_k^2 \1_{\xi_k > K_N} ]
		= 0,
	\end{equation}
	see Lemma~\ref{lem:kerstingAssumption}.
	Compare this with the weaker assumptions \eqref{eq:uniform_integrability} 
	and \eqref{eq:lindeberg}. In particular, \eqref{eq:kerstingConsequence} 
        prevents the limiting process $\sigma^2$ from jumping, which in
        turn implies that the limiting tree is binary.
\end{rem}

\subsection{Examples}
\label{sec:examples}

We now present some examples illustrating the conditions
\eqref{Cond_on_the_environment} and \eqref{assumption_variance}, all of
which fit the class of models we are considering here.

\begin{enumerate}
        \item Consider a sequence of nearly critical Galton--Watson
            processes $\big(Z^{(N)}_n, n \geq 0\big)$ with offspring
            distribution $f^{(N)}$ such that
	\begin{align}
		[f^{(N)}]'(1) = 1+ \frac{\alpha}{N} + o\left(\frac{1}{N}\right) \quad  \text{and} \quad
		[f^{(N)}]''(1) = \sigma^2 + o(1),
	\end{align}
        with $\alpha \in \R$ and $\sigma^2>0$. Then
	\begin{align}
		\left(\mu_{sN}, s \geq 0 \right)=\left( \exp \left(\sum_{k=1}^{sN} \log(1+ \alpha /N + o(1/N))  \right) ,s\geq 0\right)\to \left(e^{\alpha s }, s \geq 0 \right),
	\end{align}
        in the Skorohod topology as $N \to \infty$, and the variance process
	\eqref{assumption_variance} converges to $\sigma^2(t) = \sigma^2 t$.

	\item 
        In \cite{BoKer21} the authors considered a sequence of branching
        processes in random environment such that the random first and
        second moment are given by independent copies of
	\begin{align}
		[F^{(N)}]'(1)= 1+ \frac{\alpha}{N} + \frac{1}{\sqrt{N}} \zeta, \qquad [F^{(N)}]''(1)= \sigma^2 + o(1),
	\end{align}
        with $\alpha \in \R$, $\sigma^2>0$ and some mean-zero random
        variable $\zeta$ with finite second moment. Then
	\begin{align}
		\left(\mu_{sN}, s \geq 0 \right)=\left( \exp \left( \sum_{k=1}^{sN} \log [F^{(N)}_k]'(1)\right), s\geq 0 \right) \to \left( \exp(\alpha' s + B_{\rho s}), s\geq 0 \right)
	\end{align}
        in law in the Skorohod topology as $N \to \infty$, where
        $(B_s,s\geq 0)$ is a standard Brownian motion and $\alpha', \rho$
        are some constants depending on the moments of $\zeta$. 

        \item The setting of the second example is easily adapted to
            cases in which the corresponding random walk converges
            to a general L\'evy process instead. 
	
	\item Let $p \in (0,1)$. Consider a varying environment
            such that 
	\[
	\forall k \in \tfrac{N}{2} + [0, N^{1-p}],\quad 
	\PP(\xi_k = i) =   
	\begin{cases}
		\tfrac{1}{N^p}, &\text{if $i = N^p$} \\
		1-\tfrac{1}{N^p}, &\text{if $i = 0$.}
	\end{cases}
	\]
        Elsewhere, suppose that the environment is constant with mean $1$
        and variance $\beta$. Then the variance process converges to
        the linear function $\beta t$ with a jump of size $1$ at $1/2$.
        Note that \eqref{eq:lindeberg} is fulfilled, whereas
        \eqref{eq:uniform_integrability} holds only for $p < 1/2$.
        By the same token, one can extend the construction to generate
        limiting varying environments for which $\sigma^2$ is a
        subordinator and $X$ an arbitrary independent L\'evy process.
\end{enumerate}

\subsection{Main results}
\label{sec:mainResults}

We now discuss our two main results, namely the estimate for the survival
probability and the scaling limit of the genealogy of nearly critical
BPVEs.

\begin{thm}[Kolmogorov estimate] \label{thm:kolmogorov}
    Let $(Z_n,n\geq 0)$ be a sequence of BPVEs satisfying
    \eqref{Cond_on_the_environment}, \eqref{assumption_variance}, and
    \eqref{eq:uniform_integrability}. For any $t > 0$ such that
    $\sigma^2$ is continuous and strictly increasing at $t$,
    \begin{align}
            \PP(Z_{tN} > 0) \sim \frac{2}{\kappa_N} 
            \frac{1}{\int_{[0,t]} e^{-X_s} \sigma^2(ds)},
            \qquad \text{as $N \to \infty$.}
    \end{align}
\end{thm}
We note that in the case of critical or nearly critical Galton--Watson
processes, Theorem~\ref{thm:kolmogorov} reduces to well-known
results, see \cite{Kolmogorov1938} and \cite{OConnell1995}. The proof of
Theorem~\ref{thm:kolmogorov} will be given in
Section~\ref{sec:kolmogorov_proof}.	

Our second main result, Theorem~\ref{thm:yaglom}, shows that the genealogical
structure of the BPVE, when envisioned as a random metric measure space
and conditional on survival, converges in the
Gromov--Hausdorff--Prohorov (GHP) topology to a limiting continuous tree, the
environmental coalescent point process. Stating the result requires some
preliminary notation.

Recall the tree construction of the branching process in varying
environment from Section~\ref{Sec:Model}, and the notation $T_n$ for
the $n$-th generation of the process. We define a distance on the
vertices at generation $n$ via
\[
    \forall u,v \in T_n,\quad d_n(u,v) = n - \abs{u \wedge v},
\]
which is equal to the time to the most-recent common ancestor of
$u$ and $v$. Consider the empirical measure 
\[
    \lambda_n = \sum_{u \in T_n} \delta_u.
\]
Then $[T_n, d_n, \lambda_n]$ is a \emph{random metric measure space}
that encodes the genealogy of generation $n$ of the BPVE. Formally, we
can define a topology on the space of metric measure spaces, the
Gromov--Hausdorff--Prohorov topology \cite{Abraham2013}, and $[T_n, d_n,
\lambda_n]$ is a random element of this space. The definition of this
topology is recalled in Section~\ref{sec:gromovTopologies}.

In order to state our main result we briefly introduce here the scaling limit
of the latter metric measure space, and name it the environmental
coalescent point process (CPP). We construct it as a time-change of the standard Brownian CPP \cite{Popovic2004}. A more direct
construction is provided in Section~\ref{sec:CPP}. Consider a Poisson
point process $P$ on $[0,\infty) \times (0,1]$ with intensity measure $ds
\otimes \tfrac{1}{x^2} dx$. Define a distance on $\R_+$ as
\begin{align} 
    \forall  x \leq y, \quad d_B(x,y) = \sup \{ z \in (0,1] : (s,z)\in P,\: x \leq s \leq y \}.
\end{align}
Consider an independent standard exponential random variable $Z_B$. The
Brownian CPP is the random metric measure space $[(0,Z_B), d_B, \Leb]$,
where $\Leb$ is the Lebesgue measure on the interval $(0, Z_B)$. It
corresponds to the universal scaling limit of genealogies of critical
branching processes with finite variance and is illustrated in
Figure~\ref{fig:cpp_simulation}.

\begin{figure}
	\centering
	\includegraphics[width=.9\textwidth]{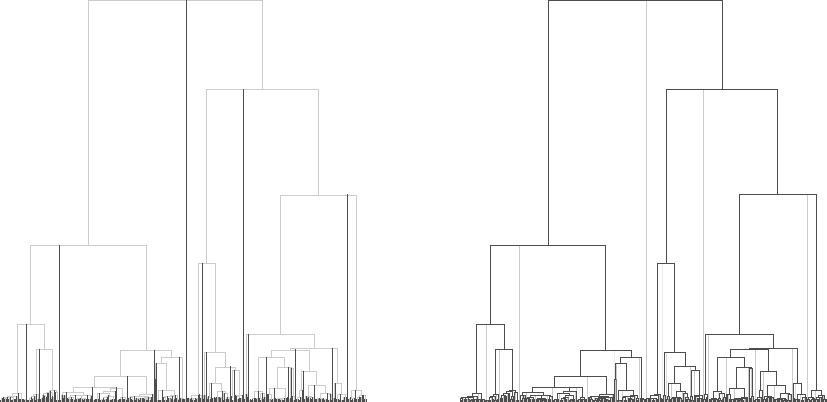}
        \caption{Simulation of a Brownian CPP. Left, a vertical line
            with height $x$ at horizontal location $s$
            represents an atom $(s,x)$ of $P$. Right, the tree
            corresponding to the CPP on the left. The distance
            $d_B$ is the tree distance on the leaves.}
	\label{fig:cpp_simulation}
\end{figure}

The environmental CPP is now defined as a time-changed version of the
Brownian CPP, which reflects the heterogeneity of means and variances due
to the varying environment. Fix a time $t > 0$ and define 
\begin{equation} \label{eq:definition_rho}
	\forall s \le t,\quad 
        F(s) \coloneqq \frac{1}{2\bar{\rho}_t} \int_{[t-s,t]} e^{X_t-X_u} \sigma^2(du),
        \qquad
	\bar{\rho}_t \coloneqq \frac{1}{2} \int_{[0,t]} e^{X_t-X_u} \sigma^2(du).
\end{equation}
Let $F^{-1}$ be the right-continuous inverse of $F$. The environmental
CPP is defined as the random metric measure space  
\[
    [(0,Z_e), d_{\nu_e}, \Leb] \coloneqq [(0, Z_B), F^{-1} \circ d_B, \bar{\rho}_t \Leb].
\]
Note that going from the Brownian to the environmental CPP requires to
rescale time according to $F$ but also to rescale mass (that is,
population size) by $\bar{\rho}_t$.

\begin{thm} \label{thm:yaglom}
    Suppose that \eqref{Cond_on_the_environment},
    \eqref{assumption_variance} and \eqref{eq:uniform_integrability}
    hold and fix $t > 0$ that fulfils the same assumption as in
    Theorem~\ref{thm:kolmogorov}. Conditional on survival at time
    $tN$, the following convergence holds in distribution for the
    Gromov--Hausdorff--Prohorov topology
    \begin{align}
            \lim_{N \to \infty} \left[ T_{tN}, \tfrac{d_{tN}}{N} ,
            \tfrac{\lambda_{tN}}{\kappa_N} \right]= \left[(0,Z_e), d_{\nu_e}, \Leb \right],
    \end{align}
    with $[ (0,Z_e), d_{\nu_e}, \Leb]$ the environmental CPP.
\end{thm}

This result is proved in Section~\ref{Sec:Proofs}. We point out
the following fact. The limiting genealogy, the environmental CPP, is
obtained by time-changing the Brownian CPP according to the function $F$
which might be discontinuous (as in the last example of
Section~\ref{sec:examples}). As a consequence, although the Brownian CPP
is a binary tree, the environmental CPP might have (simultaneous)
multiple mergers. These multiple mergers correspond to an accumulation of
binary branch points over a time-scale shorter than $N$, due to a region
of highly variable environments. This phenomenon is illustrated
numerically in Figure~\ref{fig:multipleMergers}.

\begin{figure}
    \centering
    \includegraphics[width=.65\textwidth]{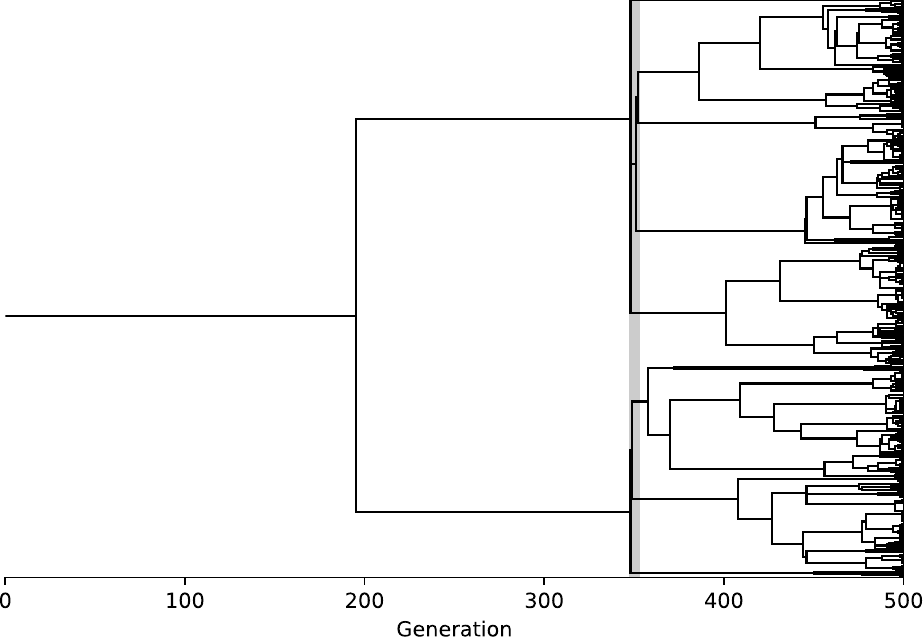}
    \caption{Simulation of the genealogy of a BPVE until generation $N =
        500$. All offspring distributions are negative binomials with
        mean $1+\alpha / N$, with $\alpha = 1$. Outside of the shaded
        region, the variance is $\sigma^2 = 2$. In the shaded region, the
        variance is increased so that the sum of the variances in
        that region is $c N$, with $c = 2$.}
    \label{fig:multipleMergers}
\end{figure}

\subsection{Consequences of \texorpdfstring{Theorem~\ref{thm:yaglom}}{Theorem 2}} 
\label{sec:consequences}

The convergence of the population in the GHP topology might seem
daunting at first sight to a reader not familiar with random metric
measure spaces. We show in this section how the GHP convergence in
Theorem~\ref{thm:yaglom} condenses the limit of several genealogical
quantities of interest. This requires some preliminary notation for these
quantities and their limits.

For each $i \le tN$, let $Z_{i, tN}$ be the number of individuals at
generation $i$ having descendants at generation $tN$ in the BPVE. Recall
that, if $(U_i)_{i \le k}$ are chosen uniformly from $T_{tN}$ (the set of
individuals alive at generation $tN$), $d_{tN}(U_i,U_j)$ denotes the time
to the most-recent common ancestor of $U_i$ and $U_j$. 
For $k \in \N$ and $\theta \ge 0$, let $(H^{\theta}_1, \dots,
H^\theta_{k-1})$ be independent random variables with cumulative
distribution function
\begin{equation} \label{eq:definition_H_theta}
	\forall s \le t, \qquad	\PP( H^{\theta} \le s ) 
        = \frac{(1+\theta) F(s)}{1+\theta F(s)}.
\end{equation}
We define an array $(H^\theta_{i,j};i,j \le k)$ via
\begin{equation} \label{eq:definition_H_max}
    \forall i \le j,\quad H^\theta_{i,j} = H^\theta_{j,i} = 
    \max \{ H^\theta_i, \dots, H^\theta_{j-1} \},
\end{equation}
with the convention that $\max \emptyset = 0$.
We note that \eqref{eq:definition_H_theta} in fact defines a valid
probability distribution for any $\theta >-1$. However, these values of
$\theta$ in $(-1,0)$ are not relevant in our case. Some motivation
for this distribution in connection with coalescent point processes is
provided in Section~\ref{sec:CPPsample}.

% For $k \in \N$ let
% $(H_1,\dots,H_{k-1})$ be a sequence of independent random variables with
% cumulative distribution function $F$ defined in
% \eqref{eq:definition_rho}. We define the array $(H_{i,j};i,j \leq k)$ via
% \begin{equation} \label{eq:definition_H_max}
% 	\forall i \le j,\quad H_{i,j} = H_{j,i} = \max \{ H_i, \dots, H_{j-1} \},
% \end{equation}
% with the convention that $\max \emptyset = 0$. Additionally, for
% $\theta\geq 0$ we introduce $H^{\theta}$ with 
% \begin{equation} \label{eq:definition_H_theta}
% 	\forall s \le t, \qquad	\PP( H^{\theta} \le s ) 
%         = \frac{(1+\theta) F(s)}{1+\theta F(s)}.
% \end{equation}
% Lastly, we introduce the array $(H^{\theta}_{i,j};i,j \leq k)$ in the
% same manner as in \eqref{eq:definition_H_max}. We note that \eqref{eq:definition_H_theta} in fact defines a valid probability distribution for any $\theta >-1$. However, these values of $\theta$ \blue{in $(-1,0)$} are not relevant in our case. Some motivation for this distribution in connection with coalescent point processes is provided in
% Section~\ref{sec:CPPsample}.

\begin{cor} \label{cor:three_consequences}
    Let $t > 0$ satisfy the same assumptions as in Theorem~\ref{thm:kolmogorov}.
    Conditional on $Z_{Nt} > 0$, the following limits hold in distribution as $N \to \infty$.
    \begin{enumerate}
        \item[(i)] We have
        \[
            \frac{Z_{Nt}}{\kappa_N}  
            \to 
            \mathscr{E}( \bar{\rho}_t ),
        \]
        where $\mathscr{E}(x)$ is an exponential random variable with mean $x > 0$.

        \item[(ii)] For the reduced process we have
        \[
            (Z_{sN, tN}, s < t) \to
            \left( Y\Big(\log \frac{\int_{[0,t]} e^{-X_u}
                \sigma^2(du)}{\int_{[s,t]} e^{-X_u} \sigma^2(du)} \Big), s < t \right)
        \]
        in the sense of finite-dimensional distribution at continuity
        points of $\sigma^2$, where $(Y(t), t \ge 0)$ is a Yule
        process. If $\sigma^2$ is continuous, the convergence holds in
        the Skorohod sense.

    \item[(iii)] Recall that $d_{tN}(U_i, U_j)$ denotes the time to the
        most-recent common ancestor of two individuals $U_i$ and $U_j$
        sampled uniformly from the population at time $tN$. Then
        \[
            \big( \tfrac{d_{tN}(U_i, U_j)}{N}\big)_{i,j}
            \to
            (\widetilde{H}_{\sigma_i,\sigma_j})_{i,j},
        \]
        where $\sigma$ is a uniform permutation of $\{1, \dots, k\}$
        and $(\widetilde{H}_{i,j})_{i,j}$ and independent array
        distributed as
        \begin{equation} \label{eq:mixture_H}
            \E\big[ \phi\big( (\widetilde{H}_{i,j})_{i,j} \big) \big]
            = k\int_0^\infty \frac{1}{(1+\theta)^2}
            \Big(\frac{\theta}{1+\theta}\Big)^{k-1}
            \E\big[\phi\big( (H^{\theta}_{i,j})_{i,j} \big)\big]
            d\theta,
        \end{equation}
        where $\big( (H^{\theta}_{i,j})_{i,j} \big)$ is defined in
        \eqref{eq:definition_H_max} and $\phi \colon \R_+^{k\times k}
        \to \R_+$ is continuous bounded.
    \end{enumerate}
\end{cor}
In all three items of the previous result, the limiting random variable
can be expressed as the image of the environmental CPP under an
appropriate functional. The previous result follows from Theorem~\ref{thm:yaglom} 
by showing that these functionals are continuous and computing the law of
the corresponding images of the limiting tree.

Finally, we illustrate our result by specifying the previous
corollary to our first example, a nearly critical Galton--Watson process
with mean $1+\alpha/N + o(1/N)$ and variance $\sigma^2 + o(1)$.
Theorem~\ref{thm:kolmogorov} becomes
\[
\pp{Z_{tN} > 0} \sim 
\begin{dcases}
	\frac{1}{N}\frac{2}{\sigma^2t}, \quad &\alpha =0 \\
	\frac{1}{N} \frac{2 \alpha }{\sigma^2 (1- e^{-\alpha t})}, \quad &\alpha \neq 0.
\end{dcases}
\]
Note that the case $\alpha=0$ corresponds to the classical Kolmogorov
asymptotics \cite{Kolmogorov1938}, whereas the case $\alpha\neq 0$ can be
found in \cite{OConnell1995}. Furthermore, point (i) of
Corollary~\ref{cor:three_consequences} yields that, conditional on
$Z_{tN} > 0$, 
\[
    \frac{Z_{tN}}{N}
    \to
    \begin{dcases}
            \mathscr{E}\big(\tfrac{t \sigma^2}{2}\big), \quad &\alpha = 0, \\
            \mathscr{E}\big(\tfrac{\sigma^2}{2} \tfrac{e^{\alpha t} - 1}{\alpha} \big), 
            \quad &\alpha \ne 0.
    \end{dcases}
\]
These are well-known asymptotics for nearly critical branching processes,
see for instance \cite{OConnell1995}. Point (iii) of
Corollary~\ref{cor:three_consequences} together with a small calculation shows that
the joint limiting distribution of the split times $s_1,\dots,s_{k-1}\in
[0,t]^{k-1}$ of the subtree spanned by $(U_1,\dots,U_k)$ has density $f$
given by
\begin{equation}
	f(s_1,s_2,\dots,s_{k-1})= 
	\begin{dcases} k \left(
		\frac{1}{t}\right)^{k-1}\int_0^\infty \frac{\theta^{k-1}}{(1+\theta)^2} 
		\prod_{i=1}^{k-1} \frac{1}{1+\theta \frac{s_i}{t}} d\theta , \quad &\alpha=0, \\
		k(\alpha( e^{\alpha t}-1))^{k-1} \int_0^\infty \frac{\theta^{k-1}}{(1+\theta)^2}
		\prod_{i=1}^{k-1} \frac{e^{\alpha s_i}}{(\theta (e^{\alpha s_i} -1) + e^{\alpha t}-1)^2} d \theta, 
		\quad &\alpha \neq 0,
	\end{dcases}
\end{equation}
which is exactly the limit in \cite[Theorem~3]{Harris2020} for $\alpha =
0$. Finally, point (ii) shows that the reduced process is asymptotically
distributed as the time-changed Yule process,
\begin{equation}
	\left( Y\Big(\log \frac{1-e^{-\alpha t}}{e^{-\alpha s} - e^{-\alpha t}} \Big), s < t \right),
\end{equation}
which again is a known result, see \cite[Remark 1]{OConnell1995}. Our work extends these results to the setting of nearly critical branching processes in varying environment.

\begin{rem}\label{rem:rho}
	Let
	\[
	\forall s \le t,\quad 
	G^{(N)}(s) = \frac{\rho_s^{(N)}}{\rho_t^{(N)}},\quad 
	\rho_s^{(N)} = \sum_{k=1}^{sN} \frac{1}{f_k'(1)^2}\frac{f_k''(1)}{\mu_{k-1}}.
	\]
	The main results in \cite{Kersting2022} and \cite{Harris2022} show that after the time-rescaling $u = G^{(N)}(s)$, the
	reduced process and the law of a $k$ sample from a BPVE converge to
	the corresponding expressions for the Brownian CPP. Since $G^{(N)}(s) \to
	G(s) \coloneqq \frac{\bar{\rho}_s}{\bar{\rho}_t}$, (see Lemma~\ref{conv_variance}),
	reverting the limiting time-change, that is taking $s = F^{-1}(u)$,
	yields the environmental CPP (see Proposition~\ref{prop:time_change_cpp}).
	Thus points (ii) and (iii) of Corollary~\ref{cor:three_consequences}
	are consistent with the results obtained in \cite{Kersting2022} and
	\cite{Harris2022}. Note again that we work under different settings, 
	see Remark~\ref{rem:comparisonAssumptions}.
\end{rem}

\section{Random metric measure spaces and Gromov topologies}
\label{sec:gromovTopologies}

Our main result shows the convergence of the genealogy of the BPVE,
viewed as a random metric measure space. This section introduces the
topological notions that are required to prove this result, as well as
some results regarding ultrametric spaces. We will need to work with
two distinct topologies on the space of metric measure spaces: the
Gromov-weak topology and the stronger Gromov--Hausdorff--Prohorov (GHP)
topology. 

The bulk of our work is dedicated to proving convergence in the 
Gromov-weak topology, which is introduced in Section~\ref{sec:gromovWeakDef}. 
In this topology, convergence in distribution is strongly connected to
convergence of \emph{moments} of the metric space, which are computed out
of samples from the metric space. These moments can be computed
efficiently for branching processes using the many-to-few formula of 
Section~\ref{sec:manyToFew}. Applying this method of moments requires the
reproduction laws to have moments of any order, whereas our convergence
result only requires a finite moment of order two. To overcome this
difficulty, we use a truncation procedure similar to that in
\cite{Harris2020}. Section~\ref{sec:truncationGromov} contains some
perturbation results that are used to prove that the truncated version of
the BPVE remains close in the Gromov-weak topology to the original BPVE.

Although proving convergence in distribution in the Gromov-weak topology
is made considerably easier by the method of moments, this topology is
too weak to ensure the convergence of some natural statistics of
genealogies, such as the convergence of the reduced process. This
motivates the introduction of a stronger topology in Section~\ref{sec:GHP}, 
the Gromov--Hausdorff--Prohorov topology. Once convergence is
established for the Gromov-weak topology by computing moments, we will
show how to strengthen it to a GHP convergence through a simple tightness
argument which was developed in \cite{Athreya2016}. A similar argument
should apply broadly to critical branching processes satisfying a Yaglom
law.

\subsection{Convergence in the Gromov-weak topology}

\subsubsection{The Gromov-weak topology}
\label{sec:gromovWeakDef}

We briefly recall the definition of the Gromov-weak topology and some of
its properties that are needed in this work. The reader is referred to
\cite{Greven2009, Gloede2013, gromov2007metric} for a more complete
account.

A \emph{metric measure space} is a triple $[X,d,\nu]$ such that $(X,d)$
is a complete separable metric space and $\nu$ is a finite measure on the
corresponding Borel $\sigma$-field. Given some $k \ge 1$ and a bounded
continuous map $\phi \colon [0, \infty)^{k\times k} \to \R$, we define a
functional $\Phi$ on the space of metric measure spaces, called a
\emph{polynomial}, as 
\[
\Phi\big( [X,d,\nu] \big) = \int_{X^k} \phi\big( (d(x_i, x_j))_{i,j \le
	k} \big) \nu(dx_1)\dots\nu(dx_k).
\]
The Gromov-weak topology is defined as the topology induced by the
polynomials, that is, the smallest topology making all polynomials
continuous. The Gromov-weak topology allows us to define a random metric
measure space as a random variable with values in metric measure spaces,
endowed with the Gromov-weak topology and the corresponding Borel
$\sigma$-field. The \emph{moment} of a random metric measure space
corresponding to $\Phi$ is the expectation of the corresponding
polynomial.

Recall the tree construction of the BPVE from Section~\ref{Sec:Model}, and
the notation $T_n$ for the $n$-th generation of the process. Recall that
$d_n$ stands for the tree distance defined as 
\[
\forall u,v \in T_n,\quad d_n(u,v) = n - \abs{u \wedge v}
\]
and consider the measure 
\[
\lambda_n = \sum_{u \in T_n} \delta_u.
\]
Then $[T_n, d_n, \lambda_n]$ is the random metric measure space
encoding the genealogy of generation $n$ of the BPVE.

The appeal of the Gromov-weak topology is well illustrated by the
following result, which connects convergence in distribution for the
Gromov-weak topology to the convergence of moments, see
\cite[Lemma~2.8]{Depperschmidt2019}. 

\begin{prop}[Method of moments] \label{thm:method_of_moments}
	Consider a sequence $([X_n, d_n, \nu_n], n \ge 1)$ of random metric measure
	spaces such that
	\[
	\forall \Phi,\quad \E\big[ \Phi\big( [X_n, d_n, \nu_n] \big) \big]
	\to 
	\E\big[ \Phi\big( [X, d, \nu] \big) \big] < \infty,
	\]
	for some limiting space $[X, d, \nu]$. Then if the limiting space
	further fulfils that
	\[
	\limsup_{k \to \infty} \frac{\E[ \nu(X)^k ]^{1/k}}{k} < \infty,
	\]
	the sequence $([X_n, d_n, \nu_n], n \ge 1)$ converges to $[X, d,
	\nu]$ in distribution for the Gromov-weak topology.
\end{prop}

\begin{rem}[Identification of the limit]
    Let us compare the previous result with the usual method of
    moments for real random variables.
    Let $(X_n)_n$ be a sequence of real r.v.s and assume that for
    every $k\in\N$ there exists some $m_k$ such that
    \[
        \mathbb{E}\big[X_n^k\big] \to m_k, \qquad \text{as $n \to \infty$}.
    \]
    Subject to an additional technical condition, the method of
    moments leads to two key outcomes: (1) the $m_k$'s coincide with
    the moments of an underlying random variable $X_\infty$, and (2)
    the sequence $(X_n)_n$ converges in distribution to $X_\infty$.
    This is in contrast with the result above, for
    which (1) is not guaranteed; specifically, the limiting moments are
    required to coincide with the moments of a limiting metric space.
    Notably, Proposition~\ref{thm:method_of_moments} can be enhanced by
    relaxing this requirement, as detailed in \cite{FoutelSchertzer22},
    Section 4. This improvement hinges on extending the Gromov-weak
    topology to non-separable ultrametric spaces through a de Finetti
    representation for exchangeable coalescents, as obtained in
    \cite{Foutel2021}.
\end{rem}

\subsubsection{Truncation in the Gromov-weak topology}
\label{sec:truncationGromov}

Applying the method of moments requires us to carry out a truncation of the
reproduction laws for all moments to be finite. In order to deduce 
the convergence of the original BPVE from that of its truncation,
we need to show that the two trees remain close with high probability
for an adequate choice of distance between trees. It turns out that one
can define a distance on the space of metric measure spaces that induces
the Gromov-weak topology \cite{Greven2009, Gloede2013, gromov2007metric}.
There are several formulations for such a distance, we work with the
Gromov--Prohorov distance of \cite{Greven2009, Depperschmidt2019}. It is
defined as
\[
    d_{\mathrm{GP}}\big( [X,d,\nu], [X',d',\nu']\big)
    =
    \inf_{Z, \phi, \phi'} d_\mathrm{P}\big(\nu \circ \phi^{-1}, 
    \nu' \circ (\phi')^{-1}\big),
\]
where the infimum is taken over all complete separable metric spaces $(Z,
d_Z)$ and isometric embeddings $\phi \colon X \hookrightarrow Z$ and
$\phi' \colon X' \hookrightarrow Z$, and where $\nu \circ \phi^{-1}$ is the
pushforward measure of $\nu$ by the map $\phi$, and $d_{\mathrm{P}}$
is the Prohorov distance between finite measures on $(Z, d_Z)$. The
following simple lemma will be the key to control the distance between
the truncated and the original tree.

\begin{lem} \label{Gromov Prohorov}
    Consider a metric measure space $[X,d,\nu]$. Let $X'$ be a closed
    subset of $X$, and let $d'$ and $\nu'$ be the restrictions of $d$ and
    $\nu$ to $X'$. Then
    \[
        d_{\mathrm{GP}}\big( [X, d, \nu], [X', d', \nu'] \big) 
        \le \nu(X \setminus X').
    \]
\end{lem}

\begin{proof}
    Since the canonical injection from $X'$ to $X$ is an isometry,
    \[
        d_{\mathrm{GP}}\big( [X, d, \nu], [X', d', \nu'] \big) 
        \le d_{\mathrm{P}}( \nu, \nu') 
        \le \nu(X \setminus X'),
    \]
    where $d_{\mathrm{P}}$ stands for the Prohorov distance on finite
    measures on $(X,d)$.
\end{proof}

The practical interest of the previous result lies in the following
corollary.

\begin{cor} \label{cor:coupling_convergence}
        Let $([X_n, d_n, \nu_n], n \ge 1)$ be a sequence of random metric
        measure spaces and, for $n \ge 1$, let $X'_n \subseteq X_n$ be a
        closed subset, and $d'_n$ and $\nu'_n$ be the restrictions of
        $d_n$ and $\nu_n$ to $X'_n$. Suppose that, as $n \to \infty$,
	\begin{enumerate}
		\item we have $\abs{\nu_n(X_n) - \nu_n(X'_n)} \to 0$ in probability;
		\item we have $[X'_n, d'_n, \nu'_n] \to [X, d, \nu]$ in distribution
		for the Gromov-weak topology.
	\end{enumerate}
        Then $([X_n, d_n, \nu_n])_n$ also converges to $[X,d,\nu]$ in
        distribution for the Gromov-weak topology.
\end{cor}

\begin{proof}
    By Lemma~\ref{Gromov Prohorov}, 
    \[
        \PP\big( d_{\mathrm{GP}}\big( [X_n, d_n, \nu_n], [X'_n, d'_n, \nu'_n]
        \big) \ge \epsilon \big)
        \le
        \PP\big( \nu_n(X_n) - \nu_n(X'_n) \ge \epsilon \big)
        \to 0,\qquad \text{as $n \to \infty$.}
    \]
    Since $d_{\mathrm{GP}}$ induces the Gromov-weak topology, the convergence
    of $([X_n, d_n, \nu_n])_n$ now follows from that of $([X'_n, d'_n, \nu'_n])_n$.
\end{proof}

\subsection{The Gromov--Hausdorff--Prohorov topology}
\label{sec:GHP}

We now turn to the definition of two very related topologies, the
Gromov--Hausdorff--Prohorov and the Gromov--Hausdorff-weak topologies.
Following \cite{Athreya2016}, we say that a sequence of metric measure
spaces $\big([X_n, d_n, \nu_n]\big)_{n \ge 1}$ converges in the Gromov--Hausdorff-weak 
sense to $[X, d, \nu]$ if there exists a complete separable metric space
$(Z, d_Z)$ and isometries $\phi_n \colon \supp \nu_n \hookrightarrow Z$
and $\phi \colon \supp \nu \hookrightarrow Z$ such that 
\begin{equation} \label{eq:GHwConvergence}
    d_{\mathrm{H}}\big( \phi_n(\supp \nu_n), \phi(\supp \nu) \big) 
    +
    d_{\mathrm{P}}\big( \nu_n \circ \phi_n^{-1}, \nu \circ \phi^{-1} \big) 
    \to 0,\qquad \text{as $n \to \infty$},
\end{equation}
where $d_{\mathrm{H}}$ is the Hausdorff distance on $(Z, d_Z)$ and
$d_{\mathrm{P}}$ is the Prohorov distance on finite measure on $(Z,d_Z)$. On the
other hand, this sequence converges in the Gromov--Hausdorff--Prohorov
sense if there exist isometries $\phi_n \colon X_n \hookrightarrow Z$ and
$\phi \colon X \hookrightarrow Z$ such that 
\begin{equation} \label{eq:GHPConvergence}
    d_{\mathrm{H}}\big( \phi_n(X_n), \phi(X) \big) 
    +
    d_{\mathrm{P}}\big( \nu_n \circ \phi_n^{-1}, \nu \circ \phi^{-1} \big) 
    \to 0,\qquad \text{as $n \to \infty$.}
\end{equation}
Clearly, a sequence of metric measure spaces converges
Gromov--Hausdorff-weakly if and only if the sequence of the
supports of the measures converges in the Gromov--Hausdorff--Prohorov
sense. The two notions of convergence coincide when the spaces have full
support, that is, when $\supp \nu = X$.

The Gromov--Hausdorff--Prohorov topology is fairly standard and has been
used extensively for deriving scaling limits of trees and other metric
spaces, which makes it a natural choice for formulating our main result.
However, when trying to reinforce a Gromov-weak convergence to a GHP one,
one runs into the problem that the two topologies are not formally
defined on the same state space. The GHP topology is defined on the space of
\emph{strong} equivalence classes of metric measure spaces: two 
spaces are said to be strongly equivalent if there exists a bijective
isometry between them that preserves the measures, see Definition~2.4 in
\cite{Abraham2013}. On the other hand, the Gromov-weak topology is
defined on \emph{weak} equivalence classes of metric measure spaces: two
spaces are weakly equivalent if their supports are strongly equivalent in
the above sense, see Definition~2.1 in \cite{Greven2009}. The
Gromov--Hausdorff-weak topology is an adaptation of the GHP topology to
weak equivalence classes; it is less standard, but more convenient to use
when working with moments and Gromov-weak convergence. In order to obtain
a GHP convergence, we will first prove the Gromov-weak convergence, then
reinforce it to a Gromov--Hausdorff-weak one, and finally use the fact
that the metric measure spaces considered here have full support to
deduce the GHP convergence.

In the rest of this section, we first introduce the notion of reduced
process of an ultrametric space and show that this reduced process is a
continuous functional of the GHP topology. Then we apply the results in
\cite{Athreya2016} to give a criterion for reinforcing a Gromov-weak
convergence to a Gromov--Hausdorff-weak one for ultrametric spaces.

\subsubsection{Ultrametric spaces and reduced processes} 
\label{sec:ultrametric}

All metric spaces that are considered in this work are constructed as
genealogies of population models at a fixed time horizon. They all have
the further property of being \emph{ultrametric spaces}. A metric space
$(U,d)$ is called ultrametric if it fulfils the stronger triangle inequality
\[
\forall x,y,z \in U, \quad d(x,y) \le \max \{ d(x,z), d(z,y) \}.
\]
In this section we construct the reduced process associated with a general
ultrametric space and prove that the function mapping an ultrametric
space to its reduced process is continuous in the GHP topology.

A well-known property of ultrametric spaces is that the set of open balls
of a given radius form a partition of the space, that is,
\[
    x \sim_t y \iff d(x,y) < t
\]
is an equivalence relation. We will use the notation $B_{x,t}(U) = \{ y
\in U : d(x,y) < t\}$ for the open ball of radius $t$ around $x$, and 
\[
B_t(U) = \{ B_{x,t}(U), x \in U \}
\]
for the set of balls of radius $t$ of an ultrametric space $(U,d)$. In
the genealogical interpretation of an ultrametric space, a ball of radius
$t$ corresponds to the set of individuals sharing a common ancestor at
time $t$ in the past. Thus, we can envision $B_t(U)$ as the set of
ancestral lineages of the population at time $t$ in the past. This
motivates the introduction of the process 
\[
\forall t > 0,\quad L_t(U) = \Card B_t(U)
\]
counting the number of ancestral lines, which we will refer to as the
\emph{reduced process} of the ultrametric space $(U,d)$. (Note the time
reversal compared to the usual definition of a reduced process: time is
flowing backward and not forward.)

The result proved in this section shows that the convergence of the reduced
process is a consequence of the convergence of the ultrametric spaces in
the GHP topology. Actually, this continuity property would hold under the
weaker Gromov--Hausdorff topology, defined for instance in Section~4 of
\cite{evans2006probability}, but we only state the result for the GHP
topology so as to not introduce notions that will not be required later
on.

\begin{prop} \label{thm:continuity_reduced}
        Suppose that $\big([U_n, d_n, \nu_n], \, n \ge 1\big)$ is a sequence of
        compact ultrametric measure spaces converging in the
        Gromov--Hausdorff--Prohorov topology to $[ U, d, \nu ]$. Then
	\[
	(L_t(U_n), t > 0) \to (L_t(U), t > 0),
	\quad \text{as $n \to \infty$},
	\]
        in the sense of pointwise convergence at each continuity point of
        $(L_t(U), t > 0)$. If $(U,d)$ is binary, that is, $(L_t(U), t >
        0)$ only makes jumps of size $1$, the convergence also holds in
        the Skorohod space $\mathbb{D}([t_0, \infty))$ for all continuity
        point $t_0 > 0$ of the reduced process.
\end{prop}

\begin{proof}
    Let $(Z, d_Z)$ and $\phi \colon U \hookrightarrow Z$, $\phi_n \colon
    U_n \hookrightarrow Z$ be as in \eqref{eq:GHPConvergence}. Fix $n \ge
    1$ and suppose that 
    \begin{equation} \label{eq:smallHausdorff}
        d_{\mathrm{H}}(\phi(U), \phi_n(U_n)) < \epsilon,
    \end{equation}
    for some $\epsilon > 0$. Fix $t > 2\epsilon$ and let $(x_i)_{i \le L_t(U)}
    \in U$, such that each $x_i$ belongs to a distinct open ball of
    radius $t$ of $U$. (That is, they are such that $d(x_i, x_j) \ge t$,
    for $i \ne j$.) By \eqref{eq:smallHausdorff}, for each $i$ there
    exists $y_i \in U_n$ such that $d_Z(\phi(x_i), \phi_n(y_i)) <
    \epsilon$. Moreover, 
    \[
        d_n(y_i, y_j) = d_Z(\phi_n(y_i), \phi_n(y_j))
        > d_Z(\phi(x_i), \phi(x_j)) - 2 \epsilon
        \ge t - 2 \epsilon.
    \]
    This shows that there are at least $L_t(U)$ open balls of $(U_n, d_n)$ of
    radius $t-2\epsilon$, that is,
    \[
        L_{t-2\epsilon}(U_n) \ge L_t(U).
    \]
    We deduce that \eqref{eq:smallHausdorff} implies that 
    \begin{equation} \label{eq:boundBalls}
        L_{t+2\epsilon}(U) \le L_t(U_n) \le L_{t-2\epsilon}(U).
    \end{equation}

    Let $t > 0$ be a continuity point of $(L_s(U), s > 0)$. We can find
    $\epsilon > 0$ such that $L_{t+2\epsilon}(U) = L_{t-2\epsilon}(U)$.
    Equation \eqref{eq:boundBalls} shows that, for $n$ large enough, 
    \[
        L_{t+2\epsilon}(U) = L_t(U_n) = L_{t-2\epsilon}(U),
    \]
    hence the first part of the result.

    If the process $(L_t(U), t > 0)$ only makes jumps of size $1$, the
    Skorohod convergence follows for instance from \cite[Chapter VI,
    Theorem~2.15]{Jacod2013}.
\end{proof}

\subsubsection{From Gromov-weak to GHP convergence}

In this last section, following \cite{Athreya2016}, we indicate how to
reinforce a convergence of ultrametric spaces in the Gromov-weak topology
to a convergence in the Gromov--Hausdorff-weak topology. (As discussed
earlier, in our context Gromov--Hausdorff-weak convergence will be
equivalent to the Gromov--Hausdorff--Prohorov one.)

In the general case of \cite{Athreya2016}, this relies on verifying that
the metric measure spaces fulfil the so-called lower mass-bound property,
see \cite[Definition~3.1]{Athreya2016}. For ultrametric spaces, since
balls form a partition of the space, this property turns out to have a
much simpler formulation.

\begin{prop} \label{prop:Gweak2GHP}
        Let $\big([U_n, d_n, \nu_n], \, n \ge 1\big)$ be a sequence
        of random compact ultrametric measure spaces such that
        $\supp \nu_n = U_n$. Suppose that it converges to a limit
        $[U,d,\nu]$ in distribution for the Gromov-weak topology. If, for
        all $t > 0$, the collection of random variables 
	\begin{equation} \label{eq:ballMass}
		\max \{ \nu_n(B)^{-1} : B \in B_t(U_n) \}, \quad n \ge 1
	\end{equation}
        is tight, then the sequence also converges in distribution for the
        Gromov--Hausdorff--Prohorov topology.
\end{prop}

\begin{rem}
    In this result, $([U_n, d_n, \nu_n])_n$ should be interpreted as a
    random sequence of strong isometry classes such that the corresponding
    sequence of weak isometry classes converges in distribution in the
    Gromov-weak topology to $[U, d, \nu]$. The limit of $([U_n, d_n,
    \nu_n])_n$ in the GHP topology is the unique strong isometry class
    with full support whose weak isometry class is $[U, d, \nu]$.
\end{rem}

\begin{proof}[Proof of Proposition~\ref{prop:Gweak2GHP}]
    We first prove that the convergence holds in the Gromov--Hausdorff-weak 
    sense, for which it is sufficient to prove that the sequence
    $\big([U_n, d_n, \nu_n], \, n \ge 1\big)$ is tight in the
    Gromov--Hausdorff-weak topology. Fix a decreasing sequence $(t_p)_{p
    \ge 1}$ such that $t_p \to 0$. For any $\epsilon > 0$ and $p \ge 1$,
    using \eqref{eq:ballMass} we can find $\eta_p > 0$ such that
    \[
        \inf_{n \ge 1} \PP( \forall B \in B_{t_p}(U_n),\: \nu_n(B) \ge
        \eta_p ) \ge 1 - \epsilon 2^{-p}.
    \]
    By the convergence of $\big([U_n, d_n, \nu_n], n \ge 1\big)$, we can also find
    a compact set (for the Gromov-weak topology) such that 
    $\PP( [U_n,d_n,\nu_n] \in K ) \ge 1 - \epsilon$ for all $n \ge 1$.
    We claim that 
    \[
        K' \coloneqq K \cap \bigcap_{p \ge 1} \{ [U',d',\nu'] : 
            \forall B \in B_{t_p}(U'),\: \nu'(B) \ge \eta_p \}
    \]
    is compact in the Gromov--Hausdorff-weak topology. To see that, note
    that $K'$ satisfies the so-called global lower mass-bound property
    that
    \[
        \forall t > 0,\quad 
        \inf_{[U',d',\nu'] \in K'} \inf_{B \in B_t(U')} \nu'( B ) > 0,
    \]
    see \cite[Definition~3.1]{Athreya2016}. Since $K' \subseteq K$, any
    sequence in $K'$ admits a subsequence that converges Gromov-weakly
    and, according to \cite[Theorem~6.1]{Athreya2016}, this subsequence
    also converges Gromov--Hausdorff-weakly. This proves that $K'$ is
    Gromov--Hausdorff-weakly compact, and since 
    \[
        \inf_{n \ge 1} \PP( [U_n, d_n, \nu_n] \in K' ) \ge 1 - 2 \epsilon,
    \]
    we have shown that $\big([U_n, d_n, \nu_n], n \ge 1\big)$ is tight in the
    Gromov--Hausdorff-weak topology.

    There is a canonical injection $\iota$ from the set of weak isometry
    classes to that of strong isometry classes, which is obtained by
    choosing a representative with full support, see \cite[equation~(5.6)]{Athreya2016}. 
    Clearly, it is a homeomorphism between the Gromov--Hausdorff-weak
    topology and the GHP topology. Therefore, we have proved that 
    \[
        \iota([U_n, d_n, \nu_n]) \to \iota([U,d,\nu]), \qquad\text{as $n \to \infty$},
    \]
    in the GHP topology. Since $[U_n,d_n,\nu_n]$ has full support,
    $\iota([U_n, d_n, \nu_n]) = [U_n, d_n, \nu_n]$ and the result is
    proved.
\end{proof}

\section{The environmental coalescent point process}
\label{sec:CPP}
The limiting genealogy of a nearly critical BPVE is described by a
continuous tree called coalescent point process (CPP). This connection is
well-known for critical Galton--Watson processes, for which the limiting
object is the Brownian CPP constructed in \cite{Popovic2004}. As our main
theorem shows, a similar result holds for varying environments provided
that we consider a larger class of limiting CPPs.
In Section~\ref{sec:mainResults}, we introduced these objects as
time-changes of the Brownian CPP. We provide a more direct construction
here, and refer the reader to \cite{Popovic2004, Lambert2013,
Duchamps2018, Lambert2017} for more background on CPPs as well as for a
more careful construction of the underlying metric measure space.

Fix some height $t > 0$ and consider a Poisson point process $P$
on $[0,\infty) \times (0,t]$ with intensity measure $ds \otimes
\nu(dx)$, for some measure $\nu$ on $(0,t]$ such that
\begin{align}
	\forall \, x > 0, \quad  \nu((x, t])<\infty, \quad \nu((0, t])=\infty.
\end{align}
Define the distance $d_{\nu}$ such that
\begin{align} \label{eq:def_cpp}
	\forall \, x\leq y, \quad d_\nu(x,y) = \sup \{ z: (s,z)\in P , x \leq s \leq y \},
\end{align}
see Figure~\ref{fig:cpp_simulation} for an illustration. Consider 
a positive real $m$ and an independent exponentially distributed random variable $Z$ with mean $m$.
The CPP at height $t$ corresponding to $\nu$ and $m$ is the random metric
measure space $[(0,Z), d_\nu, \Leb]$. Its distribution will be denoted by 
$\CPP_t(\nu, m)$. Note that the Brownian CPP of
Section~\ref{sec:mainResults} has distribution
$\CPP_1(\tfrac{dx}{x^2}, 1)$.

\begin{rem}
    In the definition of a metric measure space of Section~\ref{sec:gromovWeakDef}, 
    we require the metric space to be complete and the measure to be defined
    on the corresponding Borel $\sigma$-field. Almost surely $[(0,Z),
    d_\nu]$ is \emph{not} complete, and we always work with its
    completion without further mention. The completion can be obtained
    explicitly by adding a countable number of points as described in
    \cite{Lambert2017Bis}. Moreover, the Lebesgue measure $\Leb$ is
    formally defined on the usual $\sigma$-field on the interval $(0,Z)$,
    which might be different from the Borel $\sigma$-field of $[(0,Z),
    d_\nu]$. The two $\sigma$-fields actually coincide. This issue
    has been addressed in details in \cite{Foutel2021}. For instance, 
    \cite[Lemma~1.4]{Foutel2021} shows that, if $\mathscr{I}$ is the
    usual Borel $\sigma$-field on $(0,Z)$ and $\mathscr{B}$ that induced
    by $d_\nu$, Point~(iii) of \cite[Definition~1.1]{Foutel2021} is
    fulfilled which entails that $\mathscr{I} \subseteq \mathscr{B}$ and 
    \[
        \{ y \in (0,Z) : d_\nu(x,y) < t \} \in \mathscr{I}, \quad t >
        0,\: x \in (0,Y).
    \]
    Since the topology induced by $d_\nu$ is separable, this shows that
    $\mathscr{I} = \mathscr{B}$. The Lebesgue measure is then implicitly
    extended to the completion of $[(0,Z), d_\nu]$ by giving null mass to
    the countable number of points that were added.
\end{rem}

The distribution of the rescaled genealogy of a nearly critical BPVE at
time $tN$ will converge to $\CPP_t(\nu_e, \bar{\rho}_t)$, with
$\bar{\rho}_t$ as in \eqref{eq:definition_rho} and
\begin{equation}\label{eq:env_cpp}
	\forall s \le t,\quad \nu_e( [s,t] ) = 
	\Big(\frac{1}{2}\int_{[t-s,t]} e^{X_t-X_u} \sigma^2(du)\Big)^{-1} 
	- \Big(\frac{1}{2}\int_{[0,t]} e^{X_t-X_u} \sigma^2(du)\Big)^{-1}.
\end{equation}
We will refer to this CPP as the \emph{environmental CPP}. As
will be shown in Proposition~\ref{prop:time_change_cpp}, this definition
agrees with the definition of the environmental CPP as a time-change of
the Brownian CPP given in Section~\ref{sec:mainResults}. Note that the
Brownian CPP also corresponds to the environmental CPP at height $1$ with
$X_t = 0$ and $\sigma^2(t) = \sigma^2 t$.

The rest of this section collects some properties of CPPs that are mostly
required in the proof of Corollary~\ref{cor:three_consequences}. These
properties are well-known for the Brownian CPP. The following result will
prove convenient to transfer these results from the Brownian CPP to any
environmental CPP.

\begin{prop} \label{prop:time_change_cpp}
	Let $[(0, Z_B), d_B, \Leb]$ be the Brownian CPP with distribution
	$\CPP_1\left(\tfrac{dx}{x^2}, 1\right)$. Recall the expression of
        $F$,
	\[
	\forall s \le t,\quad F(s) = \frac{1}{2\bar{\rho}_t} \int_{[t-s,t]}
	e^{X_t-X_u} \sigma^2(du),
	\]
        and let $F^{-1}$ be its right-continuous inverse. Then 
	$[(0, Z_B), F^{-1}(d_B), \bar{\rho}_t \Leb]$ is distributed as 
	$\CPP_t(\nu_e, \bar{\rho}_t)$, with $\nu_e$ defined in
	\eqref{eq:env_cpp}.
\end{prop}

\begin{proof}
	Let $P$ be the Poisson point process on $(0,\infty) \times (0,1)$ out
	of which the Brownian CPP is constructed in \eqref{eq:def_cpp}. We
	have 
	\begin{align*}
		\forall x < y \le Z,\quad F^{-1}(d_B(x,y)) 
		&= \sup \{ F^{-1}(z) : (s,z) \in P,\, x < s < y \} \\
		&= \sup \{ z' : (s,z') \in P',\, x < s < y \},
	\end{align*}
	where $P'$ is the point process obtained by scaling the second
	coordinate of $P$ according to $F^{-1}$, namely
	\[
	P' = \sum_{(s,z) \in P} \delta_{(s, F^{-1}(z))}.
	\]
	By the mapping theorem for Poisson point processes, $P'$ is again a
	Poisson point process on $(0,\infty) \times (0,t)$, with intensity
	$ds \otimes \nu'$, where $\nu'$ is the push-forward of $\tfrac{1}{x^2}
	d x$ on $(0,1)$ by $F^{-1}$. By standard properties of generalized
	inverses, $F^{-1}(u) > s \iff u > F(s)$ and 
	\[
	\nu'((s,t]) = \frac{1}{F(s)} - 1 
	= \frac{2\bar{\rho}_t}{\int_{[t-s,t]} e^{X_t-X_u} \sigma^2(du)}
	-1 = \bar{\rho}_t \nu_e((s,t]).
	\]
	Thus, $[(0,Z_B), F^{-1}(d_B), \bar{\rho}_t \mbox{Leb}]$ is the
        CPP at height $t$ corresponding to the measure
        $\bar{\rho}_t \nu_e$ and to the random population size $Z_B$. By
        a simple scaling argument, it is not hard to see that the
        distribution of $[(0,Z_B), F^{-1}(d_B), \bar{\rho}_t \Leb]$ is
        the same as that of $[(0,Z_e), d_{\nu_e}, \Leb]$, the CPP
        constructed from $\nu_e$ and $Z_e$. (Note that the total mass
        needs to be scaled so that the isometry between $[(0,Z_B),
        F^{-1}(d_B)]$ and $[(0,Z_e), d_{\nu_e}]$ is measure-preserving.)
\end{proof}

\subsection{Moments of the CPP}

For $k \ge 1$, let $(H_1, \dots, H_{k-1})$ be independent and with
cumulative distribution function $F$ defined in \eqref{eq:definition_rho}. 
As in \eqref{eq:definition_H_max}, define an array 
\begin{equation} \label{eq:definition_H_moment}
    \forall i \le j,\quad H_{i,j} = H_{j,i} = 
    \max \{ H_i, \dots, H_{j-1} \}.
\end{equation}
The moments of the environmental CPP can be computed from the law of this
array as follows.

\begin{prop} \label{Prop_Moments_CPP}
        Let $[(0,Z_e), d_{\nu_e}, \Leb]$ be the environmental CPP at time
        $t$. Then, for any bounded function $\phi$ with
        corresponding polynomial $\Phi$,
	\begin{align}
		\EE{\Phi( (0,Z_e), d_{\nu_e}, \Leb )} = \bar{\rho}_t^k k!\, \EE{\phi(
			(H_{\sigma_i,\sigma_j})_{i,j})},
	\end{align}
        where $\left((H_{i,j})_{i,j} \right)$ is given by
        \eqref{eq:definition_H_moment} and $\sigma$ is an independent
        and uniform permutation of $[k]$.
\end{prop}

\begin{proof}
	We will apply \cite[Proposition~4]{FoutelSchertzer22}. The CPP
	construction in \cite[Section~4.3]{FoutelSchertzer22} is slightly
	different from that presented here: the measure $\nu$ is defined on the
	whole real line and $Z$ is defined as the first atom of $P$ above level $t$.
	Obviously, the two constructions are equivalent as long as $\E[Z] =
	1/\nu((t, \infty))$. Therefore, using \cite[Proposition~4]{FoutelSchertzer22}
	with $\nu((t,\infty)) = 1/\bar{\rho}_t$, the result is proved if we show that 
	\[
	\forall s \le t,\quad \PP(H \le s) = 
        \frac{1/\bar{\rho}_t}{\nu_e((s,t])+1/\bar{\rho}_t}.
	\]
	This follows from the simple computation
	\[
	(\bar{\rho}_t \nu_e((s,t]) + 1)^{-1} 
        = \frac{1}{2\bar{\rho}_t} \int_{[t-s,t]} e^{X_t-X_u} \sigma^2(du).
        \qedhere
	\]
\end{proof}

\subsection{Sampling from the CPP} \label{sec:CPPsample}

In this section we are interested in the distribution of the subtree
spanned by $k$ individuals chosen uniformly from an environmental CPP.
In the context of the CPP, choosing these individuals amounts to sampling
$k$ points $(U_1,\dots, U_k)$ uniformly on the interval $(0, Z_e)$, which
splits $(0, Z_e)$ into $k+1$ subintervals. The pairwise distances between
$(U_1, \dots, U_k)$ are then connected to the maximum of the point
process $P$ used in the construction of $d_{\nu_e}$ on these $k+1$
subintervals. Computing the joint distribution of these maxima is made
difficult by the fact that the lengths of these subintervals are not
independent.

This problem would be much easier if the points were sampled according to
an independent Poisson point process with intensity $\theta / \bar{\rho}_t dt$.
The atoms of such a point process split $(0, Z_e)$ into a geometric
number of subintervals with success parameter $1/(1+\theta)$, and the
lengths of these intervals are independent exponential random variables
with mean $\bar{\rho}_t / (1+\theta)$. Thus, the distances between any two
consecutive atoms are independent, and distributed as the maximum of $P$
over an interval with exponentially distributed length and
mean $\bar{\rho}_t / (1+\theta)$. If $H^{\theta}$ denotes a random
variable with the same distribution as this maximum, a direct computation
shows that 
\begin{equation} \label{eq:maxCPPExp}
	\forall s \le t,\quad \PP( H^{\theta} \le s ) 
        = \frac{(1+\theta) F(s)}{1+\theta F(s)},
\end{equation}
where $F$ is defined in \eqref{eq:definition_rho}, 
and corresponds to the cumulative distribution function of the maximum of
the point process $P$ over the whole interval $(0, Z_e)$.
It turns out that we can express the distribution of the pairwise
distances of $k$ points sampled uniformly from the CPP as a mixture of
that for the Poisson sampling. This can be seen as an instance of the
Poissonization method developed in \cite{Lambert2018, Johnston2019B}.

\begin{prop} \label{prop:sample_cpp_unbiased}
	For $k \ge 1$, let $(U_1,\dots, U_k)$ be uniformly distributed on
	$(0, Z_e)$. Then
	\begin{equation} \label{eq:moment_mixture}
		\E\big[ \phi\big( (d_{\nu_e}(U_i,U_j))_{i,j}\big) \big]
		= k\int_0^\infty \frac{1}{(1+\theta)^2}
		\Big(\frac{\theta}{1+\theta}\Big)^{k-1}
		\E \big[\phi\big( (H^{\theta}_{\sigma_i,\sigma_j})_{i,j} \big)\big]
		d\theta,
	\end{equation}
	where $\sigma$ is a uniform permutation of $[k]$ and $\big((H^{\theta}_{i,j})_{i,j}\big)$ as in \eqref{eq:definition_H_theta}.
\end{prop}

The key point in our proof is the observation that the exponential
distribution can be expressed as a mixture of Gamma distribution.

\begin{lem} \label{lem:expGamma}
	Fix some $k \ge 2$, then
	\[
	\forall x \ge 0,\quad e^{-x}
	= k \int_0^\infty 
	\frac{1}{(1+\theta)^2} \Big(\frac{\theta}{1+\theta}\Big)^{k-1}
	f_{k+1, \theta+1}(x) d \theta,
	\]
	where $f_{k,\theta}(x)$ is the density of a $\mathrm{Gamma}(k, \theta)$
	random variable, that is,
	\[
	\forall x \ge 0,\quad f_{k,\theta}(x) = \frac{\theta^k x^{k-1}}{(k-1)!} e^{-x \theta}.
	\]
\end{lem}

\begin{proof}
	The result follows by a direct computation, 
	\begin{align}
		\int_0^\infty 
		\frac{k}{(1+\theta)^2} \Big(\frac{\theta}{1+\theta}\Big)^{k-1}
		f_{k+1, \theta+1}(x) d \theta
		&= \int_0^\infty 
		\frac{k}{(1+\theta)^2} \Big(\frac{\theta}{1+\theta}\Big)^{k-1}
		\frac{(\theta+1)^{k+1} x^{k}}{k!} e^{-x (\theta+1)}
		d \theta \\
		&= \int_0^\infty 
		\frac{\theta^{k-1} x^{k}}{(k-1)!} e^{-x (\theta+1)}
		d \theta
		= e^{-x}. \qedhere
	\end{align}
\end{proof}

\begin{rem}
	The previous result can be extended to geometric and negative
	binomial distributions, by noting that these distributions are
	mixtures of a Poisson distribution by an exponential and Gamma
	distribution respectively. This can in turn be used to compute the
	distribution of $k$ points sampled uniformly from a discrete
	coalescent point process. In particular this provides an alternative
	proof to the main result of \cite{Lambert2018}, see Theorem~3.
\end{rem}

\begin{proof}[Proof of Proposition~\ref{prop:sample_cpp_unbiased}]
	Let $[(0,Z_e), d_{\nu_e}, \Leb]$ have distribution $\CPP_t(\nu_e, \bar{\rho}_t)$, 
	let $(V_1,\dots,V_k)$ be i.i.d.\ uniform random variables on
	$(0,1)$, and define $U_i = V_i Z_e$, $i \le k$. Since $Z_e$,
	$d_{\nu_e}$, and $(V_i)_i$ are independent, using
	Lemma~\ref{lem:expGamma} we can write
	\[
	\E\big[ \phi\big( (d_{\nu_e}(U_i, U_j))_{i,j} \big) \big]
	= k \int_0^\infty 
	\frac{1}{(1+\theta)^2} \Big(\frac{\theta}{1+\theta}\Big)^{k-1}
	\E\big[ \phi\big( (d_{\nu_e}(V_i Z_e, V_j Z_e))_{i,j}\big)\big]
	d \theta,
	\]
	where $Z_e$ has the $\mathrm{Gamma}(k+1, (\theta+1)/\bar{\rho}_t)$
	distribution and all other random variables are independent.
	The variables $(U_1, \dots, U_k)$ split $(0, Z_e)$ into $k+1$
        subintervals. By well-known properties of Gamma distributions,
        the lengths of these subintervals are independent exponential
        random variables with mean $\bar{\rho}_t/(\theta+1)$. Let
        $(U^*_1, \dots, U^*_k)$ be the order statistics of
        $(U_1,\dots,U_k)$, and define
	\[
	\forall i < k,\quad H^{\theta}_i = \sup \{ x : (s, x) \in P \text{ and } s \in [U^*_i, U^*_{i+1}] \}.
	\]
	According to the previous discussion, the random variables
        $(H^{\theta}_1,\dots, H^{\theta}_{k-1})$ are i.i.d.\ and
        distributed as the maximum of the point process $P$ over an
        interval with exponentially distributed length and mean
        $\bar{\rho}_t / (1+\theta)$. The distribution of this maximum is 
	given by \eqref{eq:maxCPPExp}. By construction of the CPP distance, 
	\[
	\forall i < j,\quad d_{\nu_e}(U^*_i, U^*_j) = H^{\theta}_{i,j} \coloneqq
	\max \{ H^{\theta}_i, \dots, H^{\theta}_{j-1} \}.
	\]
	Thus,
	\[
	\E\big[ \phi\big( (d_{\nu_e}(U^*_i, U^*_j))_{i,j}\big)\big]
	= \E\big[ \phi\big( (H^{\theta}_{i,j})_{i,j} \big) \big]
	\]
	and the result is proved by noting that the unique partition $\sigma$
	of $[k]$ such that 
	\[
	\forall i \le k,\quad U^*_i = U_{\sigma_i},
	\]
	is uniform and independent of all other variables.
\end{proof}

\subsection{The reduced process of a CPP}

Recall the definition of the reduced process $(L_s, 0 < s < t)$ of an
ultrametric tree from Section~\ref{sec:ultrametric}. The process $L_s$
counts the number of ancestral lineages at time $s$, starting from the
leaves at $s=0$ and going towards the root at $s=t$. We give a
characterization of the reduced process associated with the
environmental CPP. It is usual to express the reduced process of a
branching process as a time-changed pure birth process, also known as a
Yule process. 

\begin{prop} \label{prop:reduced_process_cpp}
	Let $(L_s, 0 < s \le t)$ be the reduced process associated
        with the environmental CPP $[(0,Z_e), d_{\nu_e}, \Leb]$. Then
	\begin{equation} \label{eq:reduced_process}
		(L_{t-s}, s < t) \overset{(d)}{=} 
		\left( Y\Big(\log \frac{\int_{[0,t]} e^{-X_u}
			\sigma^2(du)}{\int_{[s,t]} e^{-X_u} \sigma^2(du)} \Big), s < t \right),
	\end{equation}
	for a standard Yule process $(Y(s), s \ge 0)$.
\end{prop}

\begin{proof}
	Consider a Brownian CPP, that is the random ultrametric space
	with distribution $\CPP_1(\tfrac{dx}{x^2}, 1)$, denoted by
        $[(0,Z_B), d_B, \Leb]$, with corresponding reduced process
	$(L'_s, s \le 1)$. It is well-known that
	\[
	(L'_{1-s}, s < 1) \overset{(d)}{=} 
	\big( Y\big( \log \tfrac{1}{1-s} \big), s < 1 \big),
	\]
	where $(Y(u), u \ge 0)$ is a Yule process, see for instance
	\cite[Lemma~5.5]{Duchamps2018} or \cite[Theorem~2.3]{OConnell1995}. By
	Proposition~\ref{prop:time_change_cpp}, $[(0, Z_B), F^{-1}(d_B),
	\bar{\rho}_t \Leb]$ is distributed as the environmental CPP
	\eqref{eq:env_cpp}. Since the identity $F^{-1}(d(x,y)) < t \iff
	d(x,y) < F(t)$ holds, the reduced process $(L_s, 0 <s \le t)$ of
	$[(0, Z_B), F^{-1}(d_B), \bar{\rho}_t \Leb]$ is 
	\[
	\forall s \le t,\quad L_s = L'_{F(s)},
	\]
        and since $L_{t-s}=L'_{F(t-s)} \overset{(d)}{=} Y(- \log F(t-s))$
        the result follows by noting that 
	\[
	\forall s \le t,\quad \frac{1}{F(t-s)} = 
	\frac{2\bar{\rho}_t}{\int_{[s, t]} e^{X_t-X_u} \sigma^2(du)}.
        \qedhere
	\] 
\end{proof}

\section{Spinal decompositions}
\label{sec:spinalDecompositions}

\subsection{Spinal decomposition and \texorpdfstring{$1$-spine}{1-spine}} 
\label{sec:1spine}

The proof of the Kolmogorov estimate (Theorem~\ref{thm:kolmogorov}) uses ideas of \cite{Lyons95} adapted to the case of a varying	environment. It relies on expressing the distribution of the whole tree structure $T$ of the branching process in terms of a distinguished lineage (the $1$-spine) on which subtrees are grafted. We refer to this type of result as a spinal decomposition, see for instance \cite{shi_2015}. A spinal decomposition for BPVEs has already been proposed for instance in \cite[Section~1.4.2]{Kersting2017Book}. Let us briefly recall it here and introduce some notation.

The spine is a line of individuals, which we think of as a set of
\emph{marked} individuals. We construct a tree $T$ with one marked
individual $v_n$ at each generation $n$. Start the population from a
single marked individual $v_0 = \emptyset$. At generation $n$, let the
marked individual have a number of offspring distributed as a size-biased
realization of $f_{n+1}$. That is, the individual on the spine has a number $\xi^*_{n+1}$
of children, with
\[
\PP(\xi^*_{n+1} = k) = \frac{k f_{n+1}[k]}{f'_{n+1}(1)}.
\]
Pick one child of the marked individual uniformly at random and mark it,
let all other children be unmarked. The other individuals at generation
$n$ reproduce independently according to the offspring distribution
$f_{n+1}$. We denote by $\QQ^*$ the distribution of the resulting tree
$T$ with the set of marked individuals $(\emptyset,v_1, v_2, \dots)$, and by $\QQ$ the
distribution of $T$. (That is, $\QQ$ is the projection of $\QQ^*$ on the
space of trees, with no distinguished vertices.) The importance of the
measure $\QQ^*$ comes from the following well-known connection between
$\QQ^*$ and $\PP$.

\begin{lem}\label{lem:1spine}
	Let $(T, (v_0, v_1, \dots))$ be distributed as $\QQ^*$. Conditional
	on the first $n$ generations of $T$, $v_n$ is uniformly distributed
	among individuals alive at generation $n$. Moreover $\QQ_n \ll \PP_n$ and 
	\[
	\forall n \ge 0,\quad \frac{d \QQ_n}{d \PP_n} = \frac{Z_n}{\mu_n},
	\]
	where $\QQ_n$ (resp.\ $\PP_n$) refers to the restriction of $\QQ$
	(resp.\ $\PP$) to the first $n$ generations.
\end{lem}
\begin{proof}
	The proof can be found in \cite{Kersting2017Book}, Lemma 1.2 therein. 
\end{proof}

\subsection{The many-to-few formula} \label{sec:manyToFew}

Before we introduce the $k$-spine tree, let us recall some facts about
discrete ultrametric trees that are needed in the proofs and in the
construction. First, a tree $\tau$ is called a (planar) ultrametric tree
with height $n \ge 1$ if all its leaves lie at generation $n$, that is,
\[
\forall u \in \tau,\quad d_u = 0 \iff \abs{u} = n,
\]
where $d_u$ denotes the out-degree of $u$. An ultrametric tree with $k$ leaves can always be encoded as a sequence of $k-1$ elements of $\{1, \dots, n\}$ giving the depth of the successive coalescence times between the leaves. More precisely, let $\ell_1, \dots,
\ell_k$ be the leaves of $\tau$, ordered such that
\[
\ell_1 < \dots < \ell_k,
\]
where $<$ is the planar (lexicographical) order of the tree. Then $\tau$
can be encoded as the sequence
\[
\mathcal{H}(\tau) = \big( n - \abs{\ell_1 \wedge \ell_2}, 
\dots, n - \abs{\ell_{k-1} \wedge \ell_k} \big),
\]
where $u \wedge v$ denotes the most-recent common ancestor of $u$ and $v$. Conversely, given a sequence $(h_1, \dots, h_{k-1}) \in \{1, \dots, n\}$
one can find an ultrametric tree $\tau$ such that $\mathcal{H}(\tau) =
(h_1, \dots, h_{k-1})$, that is, $\mathcal{H}$ is a bijection. The tree
$\mathcal{H}^{-1}(h_1,\dots,h_{k-1})$ is obtained through the discrete
CPP construction illustrated in Figure~\ref{fig:discrete_cpp}.
\begin{figure}
	\centering
	\includegraphics[width=.5\textwidth]{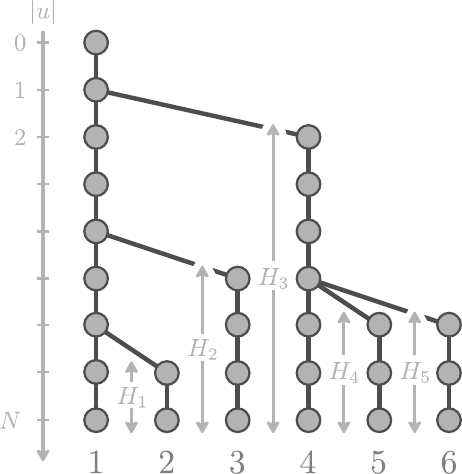}
	\caption{Illustration of the construction of the $k$-spine tree
		$S_k$, for $k=6$.}
	\label{fig:discrete_cpp}
\end{figure}
Finally, for an ultrametric tree $\tau$ we will denote by $B(\tau)$ the
set of branch points, that is,
\[
B(\tau) = \{ u \in \tau : d_u > 1 \}.
\]
Note that there are at most $k-1$ such branch points, and that we have
the identity 
\[
\sum_{u \in B(\tau)} (d_u - 1) = k-1.
\]		

We now construct the $k$-spine tree using i.i.d.\ coalescence depths in
the CPP construction. More precisely, for each $N$ consider a probability
distribution $(p_i, i \le tN)$ on $\{1, \dots, tN\}$. (As in the rest of the
paper we make the dependency on $N$ implicit.) Let $(H^{(N)}_i;\, i \ge
1)$ be a sequence of i.i.d.\ random variables with distribution $(p_i, i
\le tN)$. The $k$-spine tree $S_k$ is defined as the random ultrametric
tree $S_k= \mathcal{H}^{-1}(H^{(N)}_1, \dots, H^{(N)}_{k-1})$. As in
\eqref{eq:definition_H_max} introduce
\[
    \forall i < j,\quad H^{(N)}_{i,j} = H^{(N)}_{j,i} = \max \{ H^{(N)}_i, H^{(N)}_{i+1}, \dots,
    H^{(N)}_{j-1} \}.
\]
Following the notation in \cite{FoutelSchertzer22} we will denote the
distribution of this array by $\QQ^{k,N}$. Although $\QQ^{k,N}$ is
formally defined as a measure on $\{0, \dots, tN \}^{k\times k}$, by
construction it corresponds to the matrix of pairwise distances between
the leaves of the tree $S_k$. By abuse of notation, when convenient, we
also think of $\QQ^{k,N}$ as the law of $S_k$, that is as a measure on
ultrametric trees.

\begin{rem}
    The reader might be confused by the fact that $\Q$ denotes the law
    of the size-biased BPVE tree, whereas $\Q^{k,N}$ is the law of
    a tree $S_k$ with a fixed number $k$ of leaves. One could derive an
    extension of Lemma~\ref{lem:1spine} that connects the law $\PP$ (biased
    by its $k$-th factorial moment) to that of a tree with $k$
    distinguished lineages obtained by grafting independent subtrees on
    the spine tree $S_k$, see \cite[Theorem~3]{FoutelSchertzer22}. In
    \cite{FoutelSchertzer22} the distribution of this larger tree is
    denoted by $\bar{\QQ}^{k,N}$, and $\QQ^{k,N}$ corresponds to the
    law of the tree spanned by the $k$ distinguished lineages. We follow
    this notation for easier comparison with \cite{FoutelSchertzer22},
    even though we do not need such a precise result here.
\end{rem}

The many-to-few formula relates the factorial moments of a branching
process to the distribution of $S_k$. It involves a random bias, that we
denote by $\Delta_k$, which has the following expression in the context of
branching processes in varying environment,
\begin{equation} \label{eq:bias}
        \Delta_k \equiv \Delta_k(S_k) = k!\, \mu_{tN}^k \cdot \prod_{u \in B(S_k)} \frac{1}{(p_{tN-\abs{u}} \mu_{\abs{u}})^{d_u-1}}
	\frac{1}{d_u!} \frac{f^{(d_u)}_{\abs{u}+1}(1)}{f'_{\abs{u}+1}(1)^{d_u}},
\end{equation}
where $f^{(d_u)}_{\abs{u}+1}$ denotes the $d_u$-th derivative of the
generating function of the offspring distribution in generation $|u|$.
Note that we have included the $k!$ term in the definition of $\Delta_k$
compared to \cite{FoutelSchertzer22}. (As in the rest of the paper, we
make the dependency of $\Delta_k$ and $(p_i, i \le tN)$ on $N$ implicit,
but keep it explicit for $\QQ^{k,N}$ and $H^{(N)}_i$ for later purpose.)

\begin{prop}[Many-to-few] \label{prop:many_to_few}
        For any $k \ge 1$ and any measurable map $\phi \colon [0,
        \infty)^{k \times k} \to [0,\infty)$,
	\[
        \E\Big[ \sum_{\substack{u_1,\dots, u_k \in T_{tN}\\\text{$(u_i)$ distinct}}}
        \phi\big( (d_{tN}(u_i, u_j))_{i,j} \big) \Big]
	= \QQ^{k,N}\Big[ \Delta_k \cdot \phi\big( (H^{(N)}_{\sigma_i,\sigma_j})_{i,j}\big)
	\Big]
	\]
        where, under $\QQ^{k,N}$, $\sigma$ is an independent permutation
        of $[k]$.
\end{prop}

\begin{proof}
	The result follows from the many-to-few formula derived in
	\cite[Proposition~2]{FoutelSchertzer22} for branching processes with a
	general type space. We interpret time in a branching process with varying
	environment as a type, an individual $u$ at generation $n$ is endowed with type
	$X_u = n$. The formula in \cite{FoutelSchertzer22} requires to
	find an appropriate harmonic function for the process. Since the process
	$(Z_n / \mu_n, n \ge 0)$ is a non-negative martingale, this readily
	implies that 
	\[
	h \colon n \mapsto 1/\mu_n 
	\]
	is a harmonic function. The result now follows from \cite[Proposition~2]{FoutelSchertzer22}
	by recalling that $f^{(k)}_{n+1}(1)$ is the $k$-th factorial
	moment of the offspring distribution of an individual living at
	generation $n$. It might help the reader to note that, in the
	current work, the height of the $i$-th branch point is encoded as
	its distance to the leaves $H_i$, whereas in \cite{FoutelSchertzer22} 
	it is encoded as the distance to the root $W_i$, so that $H_i = \floor{tN} - W_i$.
\end{proof}

\section{Proof of the Kolmogorov estimate}
\label{sec:kolmogorov_proof}

We adapt the proof of Lyons, Pemantle and Peres in \cite{Lyons95} for
classical Galton--Watson processes to the case of a varying environment and make use
of the $1$-spine decomposition for BPVEs as introduced in
Section~\ref{sec:1spine}. Our proof of the Kolmogorov estimate will also
require an \emph{a priori} uniform upper bound on the probability of
survival. We derive it using the \emph{shape function} $\varphi_k(s)$ in
generation $k$ which, for $s \in [0,1)$, is defined via
\begin{align}
	\frac{1}{1-f_k(s)} = \frac{1}{(1-s) f_k'(1)} + \varphi_k(s).
\end{align}
Note that $\varphi_k(s)$ can be extended continuously to $[0,1]$ by
setting
\begin{align}
	\varphi_k(1)= \frac{f_k''(1)}{2 f_k'(1)}.
\end{align}
For more details on the shape function we refer to \cite{Kersting2020}.

\begin{lem}\label{lem:bound-proba-extinct}
    Let $t > 0$ be as in Theorem~\ref{thm:kolmogorov} and let $q_k =
    \pp{Z_{Nt}=0 \mid Z_k=1}$ be the probability of extinction in
    generation $Nt$ when starting the process in generation $k$. Let
    $\eps > 0$ be such that $\sigma^2(t)-\sigma^2(t-\epsilon) > 0$. 
    For all $k \leq (1-\eps)Nt$,
    \begin{align}
            q_k \geq 1 - c_\epsilon / \sqrt{\kappa_N},
    \end{align}
    for some constant $c_\epsilon > 0$.
\end{lem}

\begin{proof}
	By Equation (5) of \cite{BoKer21} we get
	\begin{align}
		\frac{1}{\pp{Z_{Nt}>0\mid Z_k=1}} = 
		\frac{\mu_k}{\mu_{Nt} } +
		\sum_{j=k}^{Nt-1} \frac{\mu_k \varphi_{j+1}( q_{j+1})}{\mu_j}. \label{estimate prob surv}
	\end{align}
	We aim to prove that $q_k\to 1$ as $N \to \infty$, with a
        uniform control over the speed of convergence for $k \leq
        (1-\eps) Nt$. This is equivalent to showing that the
        r.h.s.\ of \eqref{estimate prob surv} diverges uniformly. We
        have the estimate
	\begin{align}
		\sum_{j=k}^{Nt-1} \varphi_{j+1}( q_{j+1})
		&\geq \sum_{j=(1-\eps)Nt}^{Nt-1} \varphi_{j+1}(q_{j+1}) \\
		&\geq \frac{1}{2} \sum_{j=(1-\eps)Nt}^{Nt-1} \varphi_{j+1}(0),
	\end{align}
        by Lemma 1 in \cite{Kersting2020}. The r.h.s.\ above can be
        bounded from below via 
	\begin{equation*}
		\sum_{j=(1-\eps)Nt}^{Nt-1} \varphi_{j+1}(0) 
		\geq c \sqrt{\kappa_N} (\sigma^2(t) - \sigma^2(t-\epsilon)),
	\end{equation*}
	by Lemma \ref{Lemma lower bound phi}. Therefore, recalling that by
	\eqref{Cond_on_the_environment} $\mu_j$ is uniformly bounded, it follows
        that, for all $k \leq (1-\eps)Nt$,
	\begin{equation*}
		\frac{1}{1-q_k} = \frac{1}{\pp{ Z_{Nt}>0 \mid Z_k=1}} 
		\geq \sum_{j=(1-\eps)Nt}^{Nt-1} \frac{\mu_k \varphi_j(0)}{\mu_j} 
		\geq c' \sqrt{\kappa_N}
	\end{equation*}
	which immediately gives
	\begin{equation*}
		q_k \geq 1- \frac{1}{c' \sqrt{\kappa_N}}.
                \qedhere
	\end{equation*}
\end{proof}

\begin{proof}[Proof of Theorem~\ref{thm:kolmogorov}]
	In order to ease the notation, we only prove the result for $t=1$. The
	result for general $t$ follows by linear scaling.
	We follow the method of proof for estimating the survival probability
	of a critical Galton--Watson process designed in \cite{Lyons95}. Recall
	the construction of the spinal measure $\QQ^*$ in Section~\ref{sec:1spine}. Let us
 denote by $A_{N}$ the event that the spine is the left-most
	individual at generation $N$. Recalling that the spine is uniformly
	distributed among individuals at generation $N$,
	\[
	\QQ^*(Y \1_{A_{N}}) = \QQ\Big( \frac{Y}{Z_{N}} \Big) = 
	\frac{1}{\mu_{N}} \pp{Y \1_{Z_{N} > 0}}
	\]
	for any random variable $Y$ measurable w.r.t.\ the first $N$ generations
	of $T$. In particular, taking $Y = Z_{N}$ yields
	\[
	\QQ^*(Z_{N} \mid A_{N}) = \frac{\QQ^*(Z_N \1_{A_N})}{\QQ(1/Z_N)} =\frac{\mu_{N}}{\pp{ Z_{N} > 0 }}.
	\]
	We are now left with estimating the expectation on the left-hand side.
	
        We need to introduce some notation. Let us denote by $R_{N,i}$
        (resp.\ $L_{N,i}$) the number of individuals at generation $N$
        whose most-recent common ancestor with the spine lives at
        generation $i-1$, and that are to the right (resp.\ left) of the
        spine. Let $Z_{N,i} = L_{N,i} + R_{N,i}$ be the total number of
        individuals that are descended from the offspring of the spine at
        generation $i$. We have that 
	\[
            A_{N} = \bigcap_{i=1}^{N} A_{N,i}
            \coloneqq \bigcap_{i=1}^{N} \{ Z_{N,i} = R_{N,i} \},
	\]
	and the $(A_{N,i};\, i)$ are independent events. Let us also denote by
	$R_i$ the total number of individuals living at generation $i$ born to the
	right of the spine, and $\xi_i^*$ the total number of children of the spine
	at generation $i$ (including the spine individual at generation $i$). Let $v_i$ be
	the index of the spine at generation $i$.

        As a preliminary step, we notice that $R_{N,i}$ and $A_{N,i}$ are
        negatively correlated. To see this, let $Z^{(N-i)}_j$ for
        each $j \ge 1$ be the size at generation $j$ of a BPVE started
        from a single individual in generation $i$ and stopped at
        generation $N$. We get
	\begin{align}
		\EE{R_{N,i} \1_{A_{N,i}}}
                &= \EE{ \EE{ \sum_{j=1}^{R_i} Z^{(N-i)}_j q_i^{\xi^*_i-R_i-1}  \:\bigg|\: R_i, \xi^*_i} } \\
		&=\EE{ R_i \frac{\mu_N}{\mu_i} q_i^{\xi^*_i-R_i-1} } \leq \EE{ R_i \frac{\mu_N}{\mu_i} } \EE{  q_i^{\xi^*_i-R_i-1} } = \EE{R_{N,i}} \EE{\1_{A_{N,i}}}, \label{eq:negative correlation}
	\end{align}
	where we applied Chebyshev's sum inequality for the increasing function $k \mapsto k$ and the decreasing one $k \mapsto q_i^{-k}$. Now, let $R'_{N,i}$ be an independent r.v.\ such that $R'_{N,i}$ has the
        distribution of $R_{N,i}$ conditional on $A_{N,i}$. Let
        $\bar{A}_{N,i}$ denote the complement of $A_{N,i}$ and introduce
	\[
            R^*_{N,i} = R_{N,i} \1_{A_{N,i}} + R'_{N,i} \1_{\bar{A}_{N,i}}.
	\]
	Clearly, $R^*_N \coloneqq 1 + \sum_{i=1}^{N} R^*_{N,i}$ has the
	distribution of $Z_N$ under $\QQ^*(\cdot \mid A_N)$. The result will
	follow by first computing the expectation of $R_N \coloneqq 1 +
	\sum_{i=1}^N R_{N,i}$ under $\QQ^*$, then proving that the $L^1$ distance
	of $R_N$ and $R^*_N$ goes to $0$ as $N \to \infty$.
	
	Let us start by computing the expectation of $R_{N}$. The expected number
	of children of the spine individual at generation $i-1$ that are not the
	spine individual at generation $i$ is $f''_i(1)/f'_i(1)$. Each of these children
	ends up to the right of the spine with probability $1/2$, and grows an
	independent tree with expected size $\mu_{N} / \mu_i$. Therefore,
	\[
	\frac{1}{\kappa_N} \QQ^*[R_N] 
	= \frac{1}{\kappa_N} \Big(
	1 + \sum_{i=1}^N 
	\frac{f''_i(1)}{2 f'_i(1)} 
	\frac{\mu_N}{\mu_i} 
	\Big)
	\to \frac{e^{X_1}}{2} \int_{[0,1]} e^{-X_s} \sigma^2(d s),
	\]
	as $N \to \infty$, where the convergence follows from Lemma~\ref{conv_variance}.
	
	Now, let us compare $R_N$ to $R^*_N$. We first have
	\begin{align*}
		\QQ^*\Big[ \frac{1}{\kappa_N} \abs{R_N - R^*_N} \Big]
		&\le \frac{1}{\kappa_N} \sum_{i = 1}^{N} \QQ^*[R_{N,i}\1_{\bar{A}_{N,i}}] +
		\QQ^*[R'_{N,i}]\QQ^*(\bar{A}_{N,i}).
	\end{align*}
	Recall that $\QQ^*[R'_{N,i}] = \QQ^*[ R_{N,i} \mid A_{N,i} ]$. Since
	$R_{N,i}$ and $A_{N,i}$ are negatively correlated (see
        \eqref{eq:negative correlation}), we have $\QQ^*[ R_{N_i}
	\mid A_{N,i} ]\leq \QQ^*[R_{N_i}]$. In turn, $R_{N,i}$ and $\bar A_{N,i}$
	are positively correlated so that 
	\[
	\QQ^*[R'_{N,i}]\QQ^*(\bar{A}_{N,i})  
	= \QQ^*[R_{N,i}]\QQ^*(\bar{A}_{N,i}) 
	\leq \QQ^*[R_{N,i} \1_{\bar A_{N,i}}].
	\]
	Next, 
	\begin{align*}
		\QQ^*\left[R_{N,i}\1_{\bar{A}_{N,i}}\right] 
		&= \QQ^*\bigg[ \QQ^*( R_{N,i}\1_{\bar{A}_{N,i}} \mid \xi^*_i, v_i ) \bigg]  
		=  \QQ^*\bigg[  \QQ^*( R_{N,i} \mid \xi^*_i, v_i )  \QQ^*({\bar{A}_{N,i}} \mid \xi^*_i, v_i ) \bigg] \\
		&\leq \frac{\mu_N}{\mu_i} \QQ^*\bigg[  r_i  \QQ^*({\bar{A}_{N,i}} \mid \xi^*_i, v_i ) \bigg] 
		= \frac{\mu_N}{\mu_i}  \QQ^*[r_i (1-q_i^{\xi^*_i-r_i-1})] \\
		&\leq \frac{\mu_N}{\mu_i}  \QQ^*[(\xi^*_i-1) (1-q_i^{\xi^*_i-1})].
	\end{align*}
	Gathering the previous estimates yields 
	\begin{align*}
		\frac{1}{\kappa_N} \QQ^*\Big[ \abs{R_N - R^*_N} \Big] \leq \frac{2}{\kappa_N} \sum_{i = 1}^{N-1} \frac{\mu_N}{\mu_i}
		\QQ^*[(\xi^*_i-1) (1-q_i^{\xi^*_i-1})].
	\end{align*}
	Now, for $\epsilon, \epsilon' > 0$, using Lemma \ref{lem:bound-proba-extinct} and 
	Lemma~\ref{conv_variance},
	\begin{align*}
		\frac{2}{\kappa_N} &\sum_{i=1}^{N} \frac{\mu_N}{\mu_i}
                \QQ^*[(\xi^*_i-1) (1-q_i^{(\xi^*_i-1)})] \\
		&\le  \frac{2\mu_N }{\kappa_N}\Big( \sum_{i = 1}^{(1-\epsilon)N} \frac{1}{\mu_i}\QQ^*[\xi^*_i \1_{\xi^*_i \ge \epsilon' \sqrt{\kappa_N}}]
		+ \sum_{i = 1}^{(1-\epsilon)N} \frac{1}{\mu_i} \QQ^*[\xi^*_i-1] (1-q_i^{\epsilon' \sqrt{\kappa_N}})
		+ \sum_{i = (1-\epsilon)N}^{N} \frac{1}{\mu_i}  \QQ^*[\xi^*_i-1] \Big) \\
		&\le \frac{2\mu_N}{\kappa_N}\Big( \sum_{i = 1}^{N} \frac{1}{\mu_if'_i(1)} \E[(\xi^*_i)^2 \1_{\xi^*_i \ge \epsilon'\sqrt{\kappa_N}}]
		+ \sum_{i = 1}^{N} \frac{f''_i(1)}{\mu_i f'_i(1)} (1-(1-\tfrac{c_\epsilon}{\sqrt{\kappa_N}})^{\epsilon' \sqrt{\kappa_N}})
		+ \sum_{i = (1-\epsilon)N}^{N} \frac{f''_i(1)}{\mu_i f'_i(1)} \Big) \\
		&\to C\mu(1) (1-e^{\epsilon' c_\epsilon}) + C(\mu(1)-\mu(1-\epsilon)),
		\quad \text{as $N \to \infty$},
	\end{align*}
	where $\mu(t)=\int_{[0,t]} e^{-X_s} \sigma^2(ds)$
	and $C$ is a constant independent of $\eps,\eps'$.
	Since this limit can be
	made arbitrarily small by letting $\epsilon', \epsilon \to 0$, we obtain
	the desired bound on the $L^1$ distance, hence
	\begin{equation*}
		\pp{Z_N>0}= \frac{\mu_N}{\QQ^*[Z_N \mid A_N]} \sim
		\frac{\mu_N}{\kappa_N \frac{1}{\kappa_N} \QQ^*[R_N]}\sim
		\frac{2}{\kappa_N} \frac{1}{ \int_{[0,1]} e^{-X_s} \sigma^2(d s)}.
                \qedhere
	\end{equation*}
\end{proof}

\begin{rem}
        A careful reading of the argument shows that the strong Lindeberg
        condition \eqref{eq:uniform_integrability} is needed because of
        our crude estimate for the extinction probability of
        Lemma~\ref{lem:bound-proba-extinct}. If we could replace
        $\sqrt{\kappa_N}$ by $\kappa_N$ in the lemma, our Kolmogorov
        estimate would hold under the weaker condition
        \eqref{eq:lindeberg}.
\end{rem}

\section{Convergence of the BPVE} \label{Sec:Proofs}

This last section is dedicated to the proof of Theorem~\ref{thm:yaglom}. Recall the encoding of generation $tN$
of the BPVE as a random metric measure space $[T_{tN}, d_{tN},
\lambda_{tN}]$ from Section~\ref{Sec:Proofs}.

As alluded to in the introduction, we follow the general proof strategy
proposed in an earlier work by two of the authors \cite{FoutelSchertzer22}. 
We start by proving that the genealogy of the BPVE converges in the
Gromov-weak topology. According to Proposition~\ref{thm:method_of_moments}, 
proving convergence in distribution for the Gromov-weak topology amounts
to computing the limit of the moments of the random metric measure space
$[T_{tN}, d_{tN}, \lambda_{tN}]$. Using the many-to-few formula, see
Proposition~\ref{prop:many_to_few}, these moments can be expressed in
terms of the distribution of a finite tree, the $k$-spine tree
\cite{Harris2017, Harris2020, FoutelSchertzer22}. This $k$-spine tree is
studied in Section~\ref{sec:moment_computation} where we derive the limit
of the moments of the BPVE. A small caveat to this approach is that, for
the moments to be well-defined, we need to make a truncation of the
reproduction laws, which is carried out in Section~\ref{sec:truncation}.
We relate the truncated and the untruncated versions via
Corollary~\ref{cor:coupling_convergence}. Once the convergence is
established in the Gromov-weak topology, we reinforce it to a convergence
in the GHP topology by using Proposition~\ref{prop:Gweak2GHP}.

\subsection{Moments of the BPVE} 
\label{sec:moment_computation}

Our assumptions ensure that the BPVE has a finite second moment, but moments of higher are in general not finite. Thus, the method of moments
cannot be applied directly. However, we will show in
Section~\ref{sec:truncation} that, up to making a cutoff in the offspring
size, we can safely assume that there exists a sequence $\beta_N =
o(\kappa_N)$ such that 
\begin{equation} \label{eq:moment_growth}
	\forall i \ge 1,\, k\geq 2, \quad f^{(k)}_i(1) \le \beta_N^{k-2} f''_i(1),
\end{equation}
where $f^{(k)}$ denotes the $k$-th derivative of $f$. Under this
assumption, the moments of a nearly critical BPVE can be computed.

\begin{prop} \label{prop:moment_bpve}
	Consider a BPVE satisfying \eqref{Cond_on_the_environment} and
	\eqref{assumption_variance}, and a fixed $t > 0$ such that $t$ is a
	continuity point of $\sigma^2$ and $X$, and $\sigma^2$ is strictly
        increasing at $t$. If the moment growth assumption
        \eqref{eq:moment_growth} is fulfilled, then, for any bounded
        continuous $\phi$
	\[
	\frac{1}{\kappa_N^{k-1}}    
	\E\Big[ \sum_{\substack{u_1,\dots, u_k \in T_{tN}\\\text{$(u_i)$ distinct}}}
        \phi\big( \big(\tfrac{d_{tN}(u_i, u_j)}{N}\big)_{i,j} \big) \Big]
	\to 
	e^{X_t} \bar{\rho}_t^{k-1} k!\,
	\E\big[ \phi\big( (H_{\sigma_i, \sigma_j})_{i,j} \big) \big],
	\]
        where $(H_{i,j})_{i,j}$ are as in
        \eqref{eq:definition_H_moment} and $\sigma$ is an independent
        and uniform permutation of $[k]$.
\end{prop}

\begin{proof}
	Let $\abs{B_{u_1,\dots,u_k}}$ denote the number of branch points of
	the subtree spanned by $(u_1,\dots,u_k)$. We decompose the expectation
        with respect to the value of $\abs{B_{u_1,\dots,u_k}}$ as
	\[
	\E\Big[ \sum_{\substack{u_1,\dots, u_k \in T_N\\\text{$(u_i)$ distinct}}}
        \phi\big( \big(\tfrac{d_{tN}(u_i, u_j)}{N}\big)_{i,j} \big) \Big]
	= \sum_{b=1}^{k-1} \E\Big[ \sum_{\substack{u_1,\dots, u_k \in T_N\\\text{$(u_i)$ distinct}}}
        \1_{\abs{B_{u_1,\dots,u_k}}=b} \phi\big( \big(\tfrac{d_{tN}(u_i, u_j)}{N}\big)_{i,j} \big) \Big].
	\]
        There are two steps in the proof. First, we prove that, due to
        our assumption \eqref{eq:moment_growth} on the moments, the
        contributions of the terms with $b < k-1$ vanish. Second, we
        compute the limit of the terms with only binary branch points and
        show that we recover the moments of the environmental CPP.
	
	\textbf{Step 1.} Fix $b < k-1$. We apply the many-to-few formula with
	branch lengths $H_i^{(N)}$ uniformly distributed on $\{1,\dots,tN\}$. Let
	$\QQ^{k,Nt}$ be the corresponding $k$-spine tree distribution.
	By Proposition~\ref{prop:many_to_few}, we have
	\[
	\E\Big[ \sum_{\substack{u_1,\dots, u_k \in T_N\\\text{$(u_i)$ distinct}}}
        \1_{\abs{B_{u_1,\dots,u_k}}=b} \phi\big( \big(\tfrac{d_{tN}(u_i, u_j)}{N}\big)_{i,j} \big) \Big]
	= \QQ^{k, Nt}\Big[ \Delta_k \1_{\abs{B} = b} \phi\big( \big(\tfrac{H^{(N)}_{\sigma_i, \sigma_j}}{N}\big)_{i,j} \big)
	\Big],
	\]
        where $\abs{B}$ denotes the number of branch points in the $k$-spine
	tree. Recall the expression of $\Delta_k$ from \eqref{eq:bias}.
	Using assumption~\eqref{eq:moment_growth} we have
	\[
	\Delta_k = (tN)^{k-1} k!\, \mu_N^k \prod_{u \in B} 
	\frac{1}{\mu_{\abs{u}}}
	\frac{1}{d_u!}
	\frac{f^{(d_u)}_{\abs{u}+1}(1)}{f'_{\abs{u}+1}(1)^{d_u}}
	\le (tN)^{k-1} C \prod_{u \in B} \beta_N^{d_u-2} f''_{\abs{u}+1}(1),
	\]
	for some constant $C$ that depends on the uniform bounds on
	$\mu_i$ and $f'_i(1)$ from Lemma~\ref{lemma_jump_bound}. Using that
	$\sum_{u \in B} d_u - 1 = k-1$ leads to 
	\begin{align*}
		\frac{1}{\kappa_N^{k-1}} 
		\QQ^{k, Nt}\Big[ \Delta_k \1_{\abs{B} = b} \phi\big( (H^{(N)}_{\sigma_i, \sigma_j})_{i,j} \big) \Big] 
		&\le \Big(\frac{Nt}{\kappa_N}\Big)^{k-1} C \norm{\phi}_\infty 
		\QQ^{k, Nt}\Big[ \beta_N^{k-1+b} \prod_{u \in B} f''_{\abs{u}+1}(1) \1_{\abs{B} = b} \Big] \\
		&= \beta_N^{k-1+b} 
		\frac{1}{\kappa_N^{k-1}} O\Big( \Big(\sum_{j=1}^{Nt} f''_j(1) \Big)^b\Big)\\
		&= \Big(\frac{\beta_N}{\kappa_N}\Big)^{k-1+b} 
		O\Big( \Big(\frac{1}{\kappa_N} \sum_{j=1}^{Nt} f''_j(1) \Big)^b\Big) \\
		&\to 0, \quad \text{as $N \to \infty$},
	\end{align*}
	where we have used Lemma~\ref{lem:counting_trees} in the second line,
	assumption \eqref{assumption_variance} and that $\beta_N = o(\kappa_N)$ in the
	last line.

	\textbf{Step 2.} To compute the contribution of binary branch points, we
	also apply the many-to-few formula, but with a different branch length
	distribution. We assume that $H^{(N)}_j$ are independent and identically distributed with
	\[
	\forall 1 \leq  i \le tN,\quad \bar p_i = \PP(H^{(N)} = i) 
	\coloneqq \frac{1}{\bar{\rho}_{t}^{(N)}}  \frac{1}{2\kappa_N} \frac{\mu_{tN}}{\mu_{tN-i}} \frac{f''_{tN-i+1}(1)}{f'_{tN-i+1}(1)^2},
	\]
	where $\bar{\rho}_{t}^{(N)}$ is the renormalization constant making $(\bar p_i, i \le
	tN)$ a probability measure:
	\[
	\bar{\rho}_{t}^{(N)} = \frac{\mu_{tN}}{2 \kappa_N} \sum_{j = 0}^{tN-1} \frac{1}{\mu_j} \frac{f_{j+1}''(1)}{f_{j+1}'(1)^2}.
	\]
	To distinguish the distribution of this $k$-spine tree from that derived
	with uniform branch lengths, let us denote it by $\bar{\QQ}^{k,Nt}$. We
	obtain by Proposition~\ref{prop:many_to_few} that 
	\[
	\E\Big[ \sum_{\substack{u_1,\dots, u_k \in T_N\\\text{$(u_i)$ distinct}}}
        \1_{\abs{B_{u_1,\dots,u_k}}=k-1} \phi\big( \big(\tfrac{d_{tN}(u_i, u_j)}{N}\big)_{i,j} \big) \Big]
	= \bar\QQ^{k, Nt}\Big[ \Delta_k \1_{\abs{B} = k-1} 
	\phi\Big( \big(\tfrac{H^{(N)}_{\sigma_i, \sigma_j}}{N}\big)_{i,j} \Big)
	\Big],
	\]
	where $\Delta_k$ now takes the simple expression
	\[
	\Delta_k = \kappa_N^{k-1} k!\, \mu_{tN}
	\big(\bar{\rho}_t^{(N)}\big)^{k-1}.
	\]
        By Lemma~\ref{conv_variance} the distribution $H^{(N)}_i / N$
        under $\bar\QQ^{k,Nt}$ converges to the law with
        cumulative distribution function $F$ defined in
        \eqref{eq:definition_rho}.
        We want to derive the almost sure limit of $\Delta_k$ under
        $\bar\QQ^{k,Nt}$ and use the bounded convergence theorem to
        conclude the proof. Since $t$ is a continuity point of $(X_s, s
        \ge 0)$, by Lemma~\ref{conv_variance},
	\[
	\mu_{tN} \to e^{X_t} \quad \text{ and } \quad 
	\bar{\rho}_t^{(N)} \to \bar{\rho}_t
	\quad \text{as $N \to \infty$}.
	\]
	Finally, by the first part of the proof,
	\[
	\bar\QQ^{k, Nt}\Big[ \1_{\abs{B} < k-1} \Big]
	= O\Big( \frac{1}{\kappa_N^{k-1}}\bar\QQ^{k, Nt}\Big[ \Delta_k \1_{\abs{B} < k-1} \Big] \Big)
	= O\Big( \frac{1}{\kappa_N^{k-1}} \E\Big[ 
	\sum_{\substack{u_1,\dots, u_k \in T_N\\\text{$(u_i)$ distinct}}}
	\1_{\abs{B} < k-1} \Big]  \Big)
	\to 0.
	\]
	Altogether this shows that 
	\[
	\frac{1}{\kappa_N^{k-1}}\Delta_k\1_{\abs{B}=k-1} \to k!\, e^{X_t} \bar{\rho}_t^{k-1},
	\]
	and the bounded convergence theorem yields
	\[
	\frac{1}{\kappa_N^{k-1}} \bar\QQ^{k, tN}\Big[
	\Delta_k \1_{\abs{B} = k-1} \phi\Big( \big( \tfrac{H^{(N)}_{\sigma_i, \sigma_j}}{N}\big)_{i,j} \Big) \Big] 
	\to k!\, e^{X_t} \bar{\rho}_t^{k-1}
	\E\big[ \phi\big((H_{\sigma_i, \sigma_j})_{i,j}\big) \big],
	\]
	which ends the proof.
\end{proof}

\subsection{Truncating the offspring distribution}
\label{sec:truncation}

Let $(\beta_N)_N$ be a sequence of natural numbers. Let us define the
truncated offspring size $\tilde{\xi}_i \coloneqq \xi_i \wedge \beta_N$,
where $\xi_i$ denotes a generic copy of the offspring distribution in
generation $i$, hence has distribution $f_i$. In the same manner we
define $(\widetilde{Z}_i, i\geq 0)$ constructed from the truncated
variables $(\tilde{\xi}_{i,j}, i\geq 1, j \geq 1)$.
In this section, we will need the sequence of natural numbers $(\beta_N, N \ge 1)$ 
to satisfy $\beta_N = o(\kappa_N)$,
\begin{equation} \label{eq:cond_truncation}
	\sum_{j=1}^{Nt} \E[\xi_j \1_{\xi_j > \beta_N}] \to 0 \quad  \text{ and } \quad  \frac{1}{\kappa_N}\sum_{j=1}^{Nt} \E[\xi_j^2 \1_{\xi_j > \beta_N}] \to 0 
	\quad \text{as $N \to \infty$}.
\end{equation}
We know, thanks to Lemma~\ref{lem:correct_truncation}, that such a sequence
exists provided assumption \eqref{eq:lindeberg} holds. The
next result shows that the truncated BPVE satisfies all the required
properties to apply the method of moments.

\begin{lem} \label{Lemma_properties_cut_off_version}
	Consider a BPVE satisfying assumptions \eqref{Cond_on_the_environment}, 
	\eqref{assumption_variance} and \eqref{eq:lindeberg}.
	Consider a sequence $(\beta_N, N \ge 1)$ such that \eqref{eq:cond_truncation} 
	holds and the corresponding truncated r.v.\ $\tilde{\xi}_j = \xi_j
	\wedge \beta_N$, with distribution $\tilde{f}_i$. Then    
	\begin{itemize}
		\item [(i)] 
		as $N \to \infty$,
		\[
		\Big( \prod_{j=1}^{sN} \tilde{f}'_j(1), s \ge 0\Big)
		\to 
		(e^{X_s}, s \ge 0),
		\]
		in the usual Skorohod sense.
                \item [(ii)] for all $t > 0$ such that $\sigma^2$ is continuous at $t$
		\[
		\frac{1}{\kappa_N} \sum_{i =1}^{tN} \tilde{f}''_i(1) \to 
		\sigma^2(t), 
		\quad \text{as $N \to \infty$.}
		\]
	\end{itemize}
	The sequence of truncated environments also fulfils that
	\begin{equation} \label{eq:moment_bound}
		\EE{\tilde{\xi}_j^{(k)}} \leq \beta_N^{k-2} f_j''(1),
	\end{equation}
	for all $k \geq 2$, where $\EE{X^{(k)}}=\EE{X(X-1)\cdots (X-k+1)}$
	denotes the $k$-th factorial moment of a random variable $X$.
\end{lem}

\begin{proof}
	To show that \eqref{eq:moment_bound} holds, note that
	\begin{align}
		\EE{\tilde{\xi}_j^{(k)}} &= \sum_{i=k}^{\beta_N} i(i-1)\cdots (i-k+1) f_j[k] + \beta_N^{(k)} \pp{\xi_j > \beta_N} \\
		&\leq \beta_N^{k-2} \left( \sum_{i=k}^{\beta_N} i(i-1) f_j[k] + \beta_N^{2} \pp{\xi_j >\beta_N}\right) \leq \beta_N^{k-2} f_j''(1).
	\end{align}
	
        We now prove (i). By the continuous mapping theorem it is sufficient to prove
	\begin{align}
		\sum_{j=1}^{sN} \log \tilde{f}_j'(1) \to X_s
	\end{align}
	for the Skorohod topology. Therefore, consider
	\begin{align}
		\sum_{j=1}^{sN} \log \tilde{f}_j'(1)
		&=  \sum_{j=1}^{sN} \log \frac{\tilde{f}_j'(1)}{f_j'(1)}
		+ \sum_{j=1}^{sN} \log f_j'(1). \label{convergence 2 sums}
	\end{align}
	By assumption \eqref{Cond_on_the_environment}, the second term in
	\eqref{convergence 2 sums} converges to $(X_s, s \ge 0)$. Therefore, we only
	need to prove that the first term converges to $0$ uniformly. Note that
	$\tilde{f}_j'(1)\leq f_j'(1)$, hence
	\begin{align}
		\left| \sum_{j=1}^{sN} \log \frac{\tilde{f}_j'(1)}{f_j'(1)} \right| 
		&= \sum_{j=1}^{sN} \log \frac{f_j'(1)}{\tilde{f}_j'(1)} 
		\leq \sum_{j=1}^{sN} \left( \frac{f_j'(1)}{\tilde{f}_j'(1)} -1 \right)  \\
		&=\sum_{j=1}^{sN} \frac{f_j'(1)-\tilde{f}_j'(1)}{\tilde{f}_j'(1)}.
	\end{align}
	
	By \eqref{eq:cond_truncation} we have 
	\begin{align} 
		\sum_{j=1}^{sN} (f_j'(1) - \tilde{f}_j'(1)) 
		&\le \sum_{j=1}^{sN} \E[\xi_j\1_{\xi_j > \beta_N}] 
		\to 0, \quad \text{as $N \to \infty$}. \label{eq:truncation_bound}
	\end{align}
        Point (i) is thus proved if we can lower bound $\tilde{f}'_j(1)$ uniformly in
	$N$ and $j \le sN$. Write
	\[
	\tilde{f}'_j(1) = f'_j(1) - (f'_j(1) - \tilde{f}'_j(1))
	\ge f'_j(1) - \sum_{i=1}^{sN} (f'_i(1) - \tilde{f}'_i(1) ).
	\]
	The second term is going to $0$ by \eqref{eq:truncation_bound}, whereas
	$f'_j(1)$ is uniformly bounded away from $0$ by
	Lemma~\ref{lemma_jump_bound}, hence we obtain the required bound.
	
        For point (ii), notice that
	\[
	\frac{1}{\kappa_N} \sum_{j=1}^{sN} \tilde{f}''_j(1)
	= \frac{1}{\kappa_N} \sum_{j=1}^{sN} f''_j(1)
	+ \frac{1}{\kappa_N} \sum_{j=1}^{sN} (\tilde{f}''_j(1) - f''_j(1) ),
	\]
	and by assumption \eqref{assumption_variance}, we get the right limit.
\end{proof}

\begin{lem} \label{Lemma_survival_cut_off_and_normal}
    Consider a sequence of BPVEs satisfying \eqref{Cond_on_the_environment},
	and a sequence $(\beta_N, N \ge 1)$ fulfilling
        \eqref{eq:cond_truncation}. Let $\tilde{Z}_{tN}$ be the size of the
	BPVE at time $tN$ whose offspring law is truncated by $\beta_N$. Then
	\[
            \PP( Z_{tN} - \tilde{Z}_{tN} > \epsilon \kappa_N \mid Z_{tN} > 0 ) \to 0
	\quad \text{as $N \to \infty$.}
	\]
\end{lem}

\begin{proof}
    By Lemma~\ref{Lemma_properties_cut_off_version},
    \[
        \E\big[ Z_{tN} - \tilde{Z}_{tN} \big]
        = 
        \prod_{j=1}^{tN} f'_j(1) - \prod_{j=1}^{tN} \tilde{f}'_j(1)
        \to 0 \quad \text{as $N \to \infty$.}
    \]
    We obtain
	\[
            \PP( Z_{tN} - \tilde{Z}_{tN} > \epsilon \kappa_N \mid Z_{tN} > 0 ) 
            \le \frac{\E[Z_{tN} - \tilde{Z}_{tN}]}{\epsilon \kappa_N\PP(Z_{tN} > 0)}.
	\]
	Due to Theorem~\ref{thm:kolmogorov}, the denominator is
	converging to some constant and the result follows.
\end{proof}

\subsection{Proof of the Gromov-weak convergence}

We now prove that the rescaled genealogy converges in the Gromov-weak topology.
Recall that $[T_{tN}, d_{tN}, \lambda_{tN}]$ is the random metric measure
space representing the genealogy of the population at generation $tN$.

\begin{thm} \label{thm:yaglom_gromov_weak}
	Fix $t > 0$. Under the assumptions of Theorem~\ref{thm:kolmogorov} 
	the following convergence holds in the Gromov-weak topology
	\begin{align}
            \lim_{N \to \infty} \left[ T_{tN}, \tfrac{d_{tN}}{N} ,
		\tfrac{\lambda_{tN}}{\kappa_N} \right]= \left[(0,Z_e), d_{\nu_e}, \Leb \right],
	\end{align}
	with $[ (0,Z_e), d_{\nu_e}, \Leb]$ the environmental CPP as defined in \eqref{eq:env_cpp}.
\end{thm}

\begin{proof}
	By Lemma~\ref{lem:correct_truncation} there exists a sequence $\beta_N =
	o(\kappa_N)$ fulfilling \eqref{eq:moment_growth}. Consider the
	corresponding truncated BPVE $(\tilde{Z}_n, n \ge 1)$, and the
        corresponding metric measure space $[\tilde{T}_{tN},
        \tilde{d}_{tN}, \tilde{\lambda}_{tN}]$ giving its genealogy at
        time $tN$. By Lemma~\ref{Lemma_survival_cut_off_and_normal}, for
        any $\epsilon > 0$, 
	\[
            \PP(Z_{tN} - \tilde{Z}_{tN} \ge \epsilon \kappa_N \mid Z_{tN} > 0) \to 0
            \quad \text{as $N \to \infty$.}
	\]
        Applying Corollary~\ref{cor:coupling_convergence} to $[T_{tN},
        \tfrac{d_{tN}}{N}, \tfrac{\lambda_{tN}}{\kappa_N}]$ and
        $[\tilde{T}_{tN}, \tfrac{\tilde{d}_{tN}}{N}, \tfrac{\tilde{\lambda}_{tN}}{\kappa_N}]$ 
        under the measure $\PP(\: \cdot \mid Z_{tN} > 0)$ entails that the
        result is proved if we can show that $[\tilde{T}_{tN},
        \tfrac{\tilde{d}_{tN}}{N}, \tfrac{\tilde{\lambda}_{tN}}{\kappa_N}
        ]$ converges to the environmental CPP. Up to replacing $T_{tN}$
        by $\tilde{T}_{tN}$, we can therefore assume that the environment
        fulfils
        \eqref{eq:moment_growth}.
	
        By Proposition~\ref{thm:method_of_moments}, we need to show the
        following convergence, for any polynomial $\Phi$,
	\begin{align}
            \lim_{N \to \infty} \EE{\Phi(T_{tN},\tfrac{d_{tN}}{N}, \tfrac{\lambda_{tN}}{\kappa_N})}
		= \EE{\Phi((0,Z_e), d_{\nu_e}, \Leb)}. \label{convergence Gromov}
	\end{align}
	Note that due to Proposition~\ref{Prop_Moments_CPP} the right hand side
	of \eqref{convergence Gromov} translates to
	\begin{align}
		\EE{\Phi( (0,Z_e) ,d_{\nu_e}, \Leb)} 
		= k! \bar{\rho}_t^k \EE{\phi\big( (H_{\sigma_i,\sigma_j})_{i,j} \big))},
	\end{align}
        where $(H_{i,j})_{i,j}$ are given by \eqref{eq:definition_H_moment} 
        and $\sigma\colon [k]\to [k]$ is an independent random
        permutation of $[k]$. Therefore, we show that the left hand side
        in \eqref{convergence Gromov} converges to the above limit. 
	
        Using the many-to-few formula, we have shown in
        Proposition~\ref{prop:moment_bpve} that 
	\begin{align}
            \E\Big[\Phi(T_{tN}, \tfrac{d_{tN}}{N}, \tfrac{\lambda_{tN}}{\kappa_N}) \mid Z_{tN} > 0\Big]&= \frac{1}{\kappa_N \pp{Z_{tN}>0}} \frac{1}{\kappa_N^{k-1}}
		\EE{\sum_{(u_1,\dots,u_k)\in T_{tN}} \phi\big(
                \big(\tfrac{d_{tN}(u_i,u_j)}{N}\big)_{i,j}\big)} \\
		&\sim \frac{1}{\kappa_N \pp{Z_{tN}>0}} \frac{1}{\kappa_N^{k-1}}
		\EE{\sum_{\substack{(u_1,\dots,u_k)\in T_{tN}\\ u_1 \ne \dots \ne u_k}} 
                \phi\big( \big(\tfrac{d_{tN}(u_i,u_j)}{N}\big)_{i,j}\big)} \\
		&\sim \bar{\rho}_t e^{-X_t} \cdot e^{X_t} \bar{\rho}_t^{k-1}
		k!\, \EE{\phi\big( (H_{\sigma_i,\sigma_j})_{i,j} \big)},
	\end{align}
	where we applied Theorem~\ref{thm:kolmogorov} to obtain the asymptotics of the survival probability. We recover the moments of the environmental CPP and this
	completes the proof.
\end{proof}

\subsection{Completing the proof of Theorem~\ref{thm:yaglom}} 
\label{sec:completingProof}

Before completing the proof of the main result, we need the following
consequence of the Gromov-weak convergence. Fix $t > 0$ and recall the
notation $Z_{i, tN}$ for the number of individuals at generation $i$
having descendants at generation $tN$. For $j \le Z_{i,tN}$, let us
further denote by $\abs{T_{tN}}(j)$ the number of descendants of the
$j$-th individual at generation $i$ whose offspring survives until time $tN$. 

\begin{cor} \label{cor:convergence_survivors}
	Fix $s < t$. Then conditional on survival at time $tN$, 
	\[
	\Big( \frac{\abs{T_{tN}}(1)}{\kappa_N}, \dots, \frac{\abs{T_{tN}}(Z_{sN,tN})}{\kappa_N} \Big) 
	\to \big( E_1, \dots, E_K \big),
	\quad \text{as $N \to \infty$}
	\]
	in distribution, where $(E_j, j \ge 1)$ are i.i.d.\ exponential
	random variables with mean
	\[
	\frac{1}{2} \int_{[s,t]} e^{X_t-X_u} \sigma^2(du),
	\]
	and $K$ is an independent geometric random variable on $\N$ with success parameter $\frac{1}{1+c}$ 
	with 
	\begin{equation} \label{eq:geometricConstant}
		c = \frac{\int_{[0,s]} e^{-X_u} \sigma^2(du)}{\int_{[s,t]} e^{-X_u} \sigma^2(du)}.
	\end{equation}
\end{cor}

\begin{proof}
	By Theorem~\ref{thm:yaglom_gromov_weak}, conditional on survival at
	time $sN$, the population size, rescaled by $\kappa_N$, converges to
	an exponential distribution with mean $\bar{\rho}_s$. By the
	branching property, each of these individuals has a probability
	$\PP(Z_{tN} > 0 \mid Z_{sN} = 1)$ of leaving some descendants at
	time $tN$. Using Theorem~\ref{thm:kolmogorov}, since by
	assumption \eqref{Cond_on_the_environment}, for $u > s$,
	$\E[Z_{uN} \mid Z_{sN} = 1] \to e^{X_u-X_s}$, we have
	\[
	\PP(Z_{tN} > 0 \mid Z_{sN} = 1) \sim \frac{2}{\kappa_N}
	\frac{1}{\int_{[s,t]} e^{X_s-X_u} \sigma^2(du)}.
	\]
	These two facts readily imply that, conditional on survival at time
	$sN$, $Z_{sN,tN} \to K'$ in distribution with
	\[
	\forall k \ge 0,\quad \PP(K' = k) = 
	\frac{1}{1+c} \Big( \frac{c}{1+c} \Big)^k,
	\]
	where the constant $c$ is defined in  \eqref{eq:geometricConstant}.
	By the branching property, the number of descendants of each of these
	$Z_{sN,tN}$ individuals is distributed as $Z_{tN}$ conditional on
	$\{Z_{sN} = 1, Z_{tN} > 0\}$. Thus, another application of
	Theorem~\ref{thm:yaglom_gromov_weak} shows that, for any
        continuous bounded $\phi$,
	\[
	\E[ \phi(Z_{tN}) \mid Z_{tN} > 0, Z_{sN} = 1 ]
	\to 
	\E[ \phi(E_1) ],
	\quad\text{as $N \to \infty$}.
	\]
        The result is proved by noting that further conditioning on
        $\{Z_{tN} > 0\}$ only amounts to conditioning on $\{K' > 0\}$,
        which is a geometric random variable shifted by $1$, as claimed.
\end{proof}

\begin{proof}[Proof of Theorem~\ref{thm:yaglom}]
	We apply Proposition~\ref{prop:Gweak2GHP}. We only need to check that 
	\eqref{eq:ballMass} holds for the sequence of random metric measure
        spaces $([T_{tN}, \tfrac{d_{tN}}{N}, \tfrac{\lambda_{tN}}{\kappa_N}])_N$.
	In the context of the BPVE, the balls of radius $t-s$, with $s < t$,
	correspond to the ancestral lineages at time $s$. That is, each ball
	corresponds to one of the $Z_{sN,tN}$ individuals at time $sN$ having
	descendants at $tN$. Moreover, the mass of a ball is the number
	of descendants of the corresponding ancestor, rescaled by $\kappa_N$.
	In the notation of Corollary~\ref{cor:convergence_survivors}, this
	corresponds to the random variables $(\tfrac{\abs{T_{tN}}(i)}{\kappa_N}, 
	i \le Z_{sN,tN})$. By Corollary~\ref{cor:convergence_survivors}, the
	number of such balls converges to a geometric random variable, and
	the mass of the balls converge to i.i.d.\ exponential random
	variables. This readily shows that condition \eqref{eq:ballMass} is
	fulfilled, proving the result.
\end{proof}

\begin{proof}[Proof of Corollary~\ref{cor:three_consequences}]
        Both maps $[X,d,\nu] \mapsto \nu(X)$ and  $[X,d,\nu] \mapsto
        [X,d,\tfrac{\nu}{\nu(X)}]$ are continuous w.r.t.\ the Gromov-weak
        topology. Since the BPVE converges in the GHP topology by
        Theorem~\ref{thm:yaglom} this readily shows (i), and since for
        any functional $\phi$, $\E\big[ \phi\big( \big(\tfrac{d_{tN}(U_i,
        U_j)}{N}\big)_{i,j}\big) \big]$ is nothing but the moment of
        order $k$ of $[(0,Z_e), d_{\nu_e}, Z_e^{-1} \Leb]$ corresponding to
        $\phi$, Proposition~\ref{prop:sample_cpp_unbiased} proves (iii).
        The finite-dimensional part of point (ii) follows from
        Proposition~\ref{thm:continuity_reduced}. For the Skorohod part,
        note that if $\sigma^2$ is continuous, the reduced process
        of the corresponding environmental CPP only makes jump of
        size $1$ almost surely, and the result follows again from
        Proposition~\ref{thm:continuity_reduced}.
\end{proof}

% Sets up a bookmark for the acknowledgments
\bookmarksetup{startatroot}
\belowpdfbookmark{Acknowledgements}{Acknowledgements}

\section*{Acknowledgements}

We would like to thank two referees for their careful reading of our work
and their many comments, which have led to a substantial improvement of
the exposition of our results.
FFR acknowledges financial support from MfPH, the AXA Research Fund, the
Glasstone Research Fellowship, and thanks Magdalen College Oxford for a
senior Demyship.

\bibliography{bpve.bib}
\bibliographystyle{plain}

\appendix

\section{Appendix} \label{sec: Appendix}

\begin{lem} \label{lem:counting_trees}
	For each $N$ let $g_N \colon \{0, \dots, N-1\} \to [0,\infty)$ and
	let $S_k$ be the uniform ultrametric tree with $k$ leaves. 
	For any $k \ge 1$,
	\[
	\E\Big[ \prod_{u \in B(S_k)} g_N(\abs{u}) \1_{\abs{B(S_k)} = b} \Big] 
	= O_N\Big(\frac{1}{N^{k-1}} \Big( \sum_{i=0}^{N-1} g_N(i) \Big)^b \Big).
	\]
\end{lem}

\begin{proof}
	We prove the result by induction on $k$. For $k = 1$ there is no
	branch point and the claim holds. By the CPP construction, a uniform
	ultrametric tree $S_{k+1}$ with $k+1$ leaves is obtained by grafting a branch
	$k+1$ with uniform length to the right of a uniform ultrametric tree
	$S_k$ with $k$ leaves. There are two options for $S_{k+1}$ to have $b$
	branch points. Either $S_k$ had $b-1$ branch points and a new branch
	point is created. Or $S_k$ has $b$ branch points, and no new branch
	point is created. This requires that $H_k^{(N)}$ takes the same value as
	one of the variables $(H_1^{(N)}, \dots, H^{(N)}_{k-1})$, which occurs with
	probability at most $(k-1)/N$. Therefore,
	\begin{multline*}
		\E\Big[ \prod_{u \in B(S_{k+1})} g_N(\abs{u}) \1_{\abs{B(S_{k+1})} = b} \mid S_k \Big] \\ 
		\le \frac{1}{N} \sum_{i = 0}^{N-1} g_N(i)
		\cdot \prod_{u \in B(S_k)} g_N(\abs{u}) \1_{B(S_k) = {b-1}} 
		+ \frac{k-1}{N} \prod_{u \in B(S_k)} g_N(\abs{u}) \1_{B(S_k) = b} ,
	\end{multline*}
	and the result follows.
\end{proof}

\begin{lem} \label{lem:correct_truncation}
	Assume that the environment satisfies
	\eqref{eq:lindeberg}. Then for any $t > 0$ we can find
	$\beta_N = o(\kappa_N)$ such that
	\begin{equation} \label{eq:mean_truncation_conv}
		\frac{1}{\kappa_N} \sum_{j=1}^{Nt} \E[\xi_j^2 \1_{\xi_j >
			\beta_N} ] \to 0, \quad \sum_{j=1}^{Nt} \E[\xi_j \1_{\xi_j > \beta_N} ] 
		\to 0, \quad \text{as $N \to \infty$.}
	\end{equation}
\end{lem}

\begin{proof}
	We simply adapt the argument of \cite[Lemma~22]{Harris2020} to our case. By assumption \eqref{eq:lindeberg}, for any fixed $\epsilon$	
	\[
	\sum_{j=1}^{tN} \E[\xi_j \1_{\xi_j > \epsilon \kappa_N} ]
	\le \frac{1}{\epsilon \kappa_N} \sum_{j=1}^{tN} \E[\xi_j^2 \1_{\xi_j > \epsilon \kappa_N}]
	\to 0, \quad \text{as $N \to \infty$.}
	\]
	Now take $\epsilon_k \to 0$ and chose recursively $N_{k+1} > N_k$
	such that
	\[
	\forall k \ge 1,\; \forall N \ge N_k,\quad 
	\sum_{j=1}^{tN} \E[\xi_j \1_{\xi_j > \epsilon_k \kappa_N} ]< \eps_k, \quad \text{and} \quad  \frac{1}{\kappa_N} \sum_{j=1}^{Nt} \E[\xi_j^2 \1_{\xi_j > \eps_k \kappa_N} ]  < \epsilon_k.
	\]
	Set $\beta_N = \epsilon_k \kappa_N$ for $N_k \le N < N_{k+1}$. Clearly
	$\beta_N = o(\kappa_N)$ and \eqref{eq:mean_truncation_conv} is fulfilled.
\end{proof}

\begin{lem} \label{lemma_jump_bound}
	Suppose that \eqref{Cond_on_the_environment} holds. For any $t > 0$
	there exist two constants $c, C > 0$ depending only on $t$ such that 
	\[
	\forall N \ge 1, \forall i \le tN,\quad c \le f'_i(1) \le C.
	\]
	Moreover, for each continuity point $s \ge 0$ of $(X_u, u \ge 0)$ 
	and sequence $s_N \to s$,
	\[
	f'_{s_N N}(1) \to 1 \quad \text{as $N \to \infty$.}
	\]
\end{lem}

\begin{proof}
    We have
    \[
        \forall t \ge 0,\quad f'_{tN}(1) 
        = \frac{\mu_{tN}}{\mu_{tN-1}} 
        = \exp \big( \log \mu_{tN} - \log \mu_{tN - 1} \big),
    \]
    with the convention that $\mu_{-1} = 1$. By assumption
    \eqref{Cond_on_the_environment}, $(\log \mu_{tN}, t \ge 0)$ converges in
    the Skorohod topology to $(X_t, t \ge 0)$. Therefore, at every continuity
    point $s$ of $(X_t, t \ge 0)$, $f'_{s_N N}(1) \to 1$ with $s_N \to s$,
    leading to the second statement. Moreover, since $(\log \mu_{tN}, t \ge
    0)$ converges in the Skorohod topology, $(\abs{\log \mu_{tN}}, t \ge 0)$
    is uniformly bounded in $N$ on compact intervals. Taking the exponential
    leads to the first part of the claim.
\end{proof}

\begin{lem} \label{conv_variance}
	For any continuity point $t > 0$ of $\sigma^2$, and any
	bounded continuous map $\phi$, under assumptions
	\eqref{Cond_on_the_environment} and \eqref{assumption_variance},
	\[
	\frac{1}{\kappa_N} \sum_{i=0}^{tN} \phi\big(\tfrac{i}{N}\big) \frac{f''_{i+1}(1)}{f'_{i+1}(1)} \frac{1}{\mu_{i+1}}
	\to \int_{[0,t]} \phi(s) e^{-X_s} \sigma^2(d s),
	\quad \text{as $N \to \infty$.}
	\]
\end{lem}

\begin{proof}
	Let $X_N$ be distributed as 
	\[
	\forall i \le tN,\quad \PP(X_N = i) = \frac{f''_i(1)}{\sum_{j=1}^{tN} f''_j(1)}.
	\]
	The integral on the left-hand side can be expressed as 
	\[
	\Big(\frac{1}{\kappa_N}\sum_{i=1}^{tN} f''_i(1)\Big) \E\Big[ \frac{\phi(\frac{X_N}{N})}{f'_{X_N+1}(1)\mu_{X_N+1}} \Big].
	\]
	By assumption \eqref{assumption_variance}, the distribution of
	$X_N/N$ converges to the distribution with distribution function
	$s \mapsto \sigma^2(s) / \sigma^2(t)$ and the first term in the
	above expression converges to $\sigma^2(t)$. By
	Lemma~\ref{lemma_jump_bound}, for any continuity point $s$ of $(X_t, t
	\ge 0)$ and sequence $s_N \to s$, $f'_{Ns_N}(1) \to 1$ and
	by \eqref{Cond_on_the_environment} we have $\mu_{Ns_N} \to e^{X_s}$.
	Therefore using for instance \cite[Theorem~4.27]{Kallenberg2002foundations}
	proves that
	\[
	\Big(\frac{1}{\kappa_N}\sum_{i=1}^{tN} f''_i(1)\Big) \E\Big[ \frac{\phi(\frac{X_N}{N})}{f'_{X_N+1}(1)\mu_{X_N+1}} \Big]
	\to \sigma^2(t) \int_{[0,t]} \phi(x) e^{-X_s} \frac{\sigma^2(ds)}{\sigma^2(t)},
	\quad \text{as $N \to \infty$},
	\]
	yielding the claim.
\end{proof}

\begin{lem}\label{Lemma lower bound phi}
	Under the assumptions \eqref{Cond_on_the_environment},
	\eqref{assumption_variance} and \eqref{eq:uniform_integrability},
	for any fixed $t > 0$, there exists some constant $c>0$ such that 
	\begin{align}
		\forall s \le t,\quad  \sum_{j=Ns}^{Nt}
		\varphi_j(0) \geq c \sqrt{\kappa_N} \big(\sigma^2(t)-\sigma^2(s)\big),
	\end{align}
	where $\varphi_j$ is the shape function of $f_j$.
\end{lem}

\begin{proof}
	Recall the definition of the shape function
	\begin{align}
		\varphi_j(0)= \frac{1}{1-f_j(0)}- \frac{1}{f_j'(1)} = \frac{\sum_{k=2}^{\infty} f_j[k] (k-1)}{\pp{\xi_j >0} f_j'(1)}.
	\end{align}
	Fix some $\epsilon > 0$, we have
	\begin{align}
		\sigma^2(t) -\sigma^2(s) &= \lim_{N \to \infty} \frac{1}{\kappa_N}\sum_{j=Ns}^{Nt} f_j''(1) \\
		&= \lim_{N \to \infty} \frac{1}{\kappa_N} \sum_{j=Ns}^{Nt} 
		\left( \sum_{k=2}^{\epsilon \sqrt{\kappa_N}} f_j[k] k(k-1) +
		\sum_{k=\epsilon \sqrt{\kappa_N}+1}^{\infty}f_j[k](k(k-1))\right) \label{equation alpha_N}.
	\end{align}
	By assumption \eqref{eq:uniform_integrability} the second term
	vanishes. Additionally, we compute for the first term in
	\eqref{equation alpha_N}
	\begin{align}
		\sum_{j=sN}^{Nt} \sum_{k=2}^{\epsilon \sqrt{\kappa_N}} f_j[k] k(k-1) 
		&\leq \epsilon \sqrt{\kappa_N} \sum_{j=sN}^{Nt} \frac{\sum_{k=2}^{\epsilon \sqrt{\kappa_N}} f_j[k] (k-1)}{\pp{\xi_j>0}f_j'(1)} \pp{\xi_j>0}f_j'(1) \\
		&\leq \epsilon \sqrt{\kappa_N} \sum_{j=sN}^{Nt} \varphi_j(0) f_j'(1) \\
		&\leq \epsilon \sqrt{\kappa_N} \sup_{1\leq j \leq Nt} f_j'(1) \sum_{j=sN}^{Nt} \varphi_j(0).
	\end{align}
	Due to Lemma~\ref{lemma_jump_bound} $\sup_{1\leq j \leq Nt}
	f_j'(1) $ is bounded, and choosing $\epsilon$ sufficiently small
	yields the claim.
\end{proof}

Finally, consider the following assumption. For any $\epsilon > 0$,
there exists $K$ such that 
\begin{equation} \label{eq:kerstingCondition}
	\forall N \ge 1,\quad \forall k \le tN,\quad \E[\xi_k^2 \1_{\xi_k > K(1+f'_k(1))} ]
	\le 
	\epsilon \E[\xi_k^2 \1_{\xi_k \ge 2} ] < \infty.
\end{equation}
This corresponds to \cite[Assumption (B)]{Kersting2020} and
\cite[Assumption ($\ast$)]{Kersting2022}.

\begin{lem} \label{lem:kerstingAssumption}
	Consider a sequence of environments that fulfils
	\eqref{Cond_on_the_environment} and \eqref{assumption_variance}. If
	it further fulfils \eqref{eq:kerstingCondition}, then the following equivalent uniform integrability properties hold:
	\begin{itemize}
		\item[(i)] For any sequence $K_N \to \infty$,
		\[
		\lim_{N \to \infty} \frac{1}{\kappa_N} \sum_{k=1}^{tN}
		\E[\xi_k^2 \1_{\xi_k > K_N} ] = 0.
		\]
		\item[(ii)] For any $\epsilon > 0$, there exists $K$ such that 
		\[
		\lim_{N \to \infty} \frac{1}{\kappa_N} \sum_{k=1}^{tN}
		\E[\xi_k^2 \1_{\xi_k > K} ] \le \epsilon.
		\]
	\end{itemize}
\end{lem}

\begin{proof}
	It is easily seen that (i) and (ii) are equivalent. We thus only
	prove (ii). By Lemma~\ref{lemma_jump_bound}, the
	expectations $f'_k(1)$ are bounded uniformly in $N$ and $k \le Nt$,
	so that \eqref{eq:kerstingCondition} yields that for any $\epsilon >
	0$ there is $K$ such that 
	\[
	\forall N \ge 1,\quad \forall k \le tN,\quad \E[\xi_k^2 \1_{\xi_k > K} ]
	\le 
	\epsilon \E[\xi_k^2 \1_{\xi_k \ge 2} ]
	\le \epsilon \E[\xi_k^2].
	\]
	Taking a sum on both sides, and using the convergence of the variance
	\eqref{assumption_variance} leads to 
	\begin{equation*}
		\frac{1}{\kappa_N} \sum_{k=1}^{tN} \E[\xi^2_k \1_{\xi_k > K} ]
		\le \frac{\epsilon}{\kappa_N}\sum_{k=1}^{tN} \E[\xi_k^2]
		\to \epsilon \sigma^2(t) \quad \text{as $N \to \infty$}
	\end{equation*}
	yielding point (ii).
\end{proof}

\end{document}